\documentclass[a4paper,12pt,reqno]{amsart}

\usepackage[margin=2.5cm]{geometry}
\usepackage{amsthm,amsmath,amssymb}
\usepackage{dsfont}
\usepackage[english]{babel}
\usepackage[utf8]{inputenc}
\usepackage{mathtools}
\usepackage{bm}
\usepackage{stmaryrd}
\usepackage{xcolor} \definecolor{bluegreen2}{RGB}{0, 85, 127}
\usepackage[linktocpage]{hyperref}
\hypersetup{colorlinks=true, linkcolor=bluegreen2, citecolor=bluegreen2, filecolor=magenta, urlcolor=bluegreen2}
\usepackage{tikz, tikz-cd} \usetikzlibrary{arrows.meta}
\usepackage{enumitem} \setlist{itemsep=1.5mm}
\usepackage{newunicodechar}

\theoremstyle{plain}
\newtheorem{theorem}{Theorem}[section]
\newtheorem{theoremintro}{Theorem}

\newtheorem{cor}[theorem]{Corollary}

\theoremstyle{definition}
\newtheorem{defi}[theorem]{Definition}
\newtheorem{defi_intro}[theoremintro]{Definition}
\newtheorem{ex}[theorem]{Example}
\theoremstyle{remark}
\newtheorem{remark}[theorem]{Remark}

\renewcommand{\AA}{\mathds{A}}
\renewcommand{\O}{\mathcal{O}}
\renewcommand{\L}{\mathcal{L}}

\newcommand{\CC}{\mathds{C}}
\newcommand{\NN}{\mathds{N}}
\newcommand{\ZZ}{\mathds{Z}}
\newcommand{\QQ}{\mathds{Q}}

\newcommand{\PP}{\mathds{P}}
\newcommand{\LL}{\mathds{L}}
\newcommand{\M}{\mathcal{M}}
\newcommand{\set}[1]{\left\{ #1 \right\}}
\newcommand{\Mloc}{{\M}}
\newcommand{\Mhatloc}{{\widehat{\M}}}
\newcommand{\one}{\mathbf{1}}
\newcommand{\zero}{\mathbf{0}}
\newcommand{\bd}{{\mathbf{d}}}
\newcommand{\ba}{{\mathbf{a}}}
\newcommand{\Nbf}{{\bm{N}}}
\newcommand{\nubf}{{\bm{\nu}}}
\newcommand{\kbf}{{\bm{k}}}
\newcommand{\Gor}{\QQ\!\gor}
\newcommand{\KVarC}{K_0(\Var_\CC)}

\newcommand{\Zmot}{{Z_{\mot}}}

\newcommand{\Ztop}{{Z_{\topo}}}
\newcommand{\Zigusa}{{Z_{\igusa}}}
\newcommand{\llb}{\llbracket}
\newcommand{\rrb}{\rrbracket}

\DeclareMathOperator{\dd}{d\!}
\DeclareMathOperator{\ord}{ord}
\DeclareMathOperator{\Conj}{Conj}
\DeclareMathOperator{\orb}{orb}
\DeclareMathOperator{\Jac}{Jac}
\DeclareMathOperator{\Div}{div}
\DeclareMathOperator{\Sing}{Sing}
\DeclareMathOperator{\Supp}{Supp}
\DeclareMathOperator{\Var}{Var}
\DeclareMathOperator{\GL}{GL}
\DeclareMathOperator{\id}{Id}
\DeclareMathOperator{\red}{red}
\DeclareMathOperator{\reg}{reg}
\DeclareMathOperator{\sing}{sing}
\DeclareMathOperator{\gor}{Gor}
\DeclareMathOperator{\diag}{diag}
\DeclareMathOperator{\mot}{mot}
\DeclareMathOperator{\topo}{top}
\DeclareMathOperator{\igusa}{Igusa}

\title{Motivic Zeta Functions on $\mathds{Q}$-Gorenstein Varieties}

\author[E.~Le\'on-Cardenal]{Edwin Le\'on-Cardenal}
\address[E.~Le\'on-Cardenal]{CONACYT -- Centro de Investigaci\'on en Matem\'aticas (CIMAT). Unidad Zacatecas \\
Av.~Lasec Andador Galileo Galilei\\
Manzana 3 Lote 7\\
98160, Zacatecas, Mexico}
\email{edwin.leon@cimat.mx}

\author[J.~Mart\'{\i}n-Morales]{Jorge Mart\'{\i}n-Morales}
\address[J.~Mart\'in-Morales]{Centro Universitario de la Defensa, IUMA \\
Academia General Militar \\
Ctra.~de Huesca s/n. \\
50090, Zaragoza, Spain}
\urladdr{http://cud.unizar.es/martin}
\email{jorge@unizar.es}

\author[W.~Veys]{Willem Veys}
\address[W.~Veys]{
University of Leuven (KU Leuven),
Department of Mathematics,
Celestijnenlaan 200B, B-3001
Leuven (Heverlee), Belgium}
\urladdr{https://perswww.kuleuven.be/wim\_veys}
\email{wim.veys@kuleuven.be}

\author[J.~Viu-Sos]{Juan Viu-Sos}
\address[J.~Viu-Sos]{
	IMPA - Instituto de Matem\'atica Pura e Aplicada,
	Estr. Dona Castorina, 110 - Jardim Bot\^anico,
	Rio de Janeiro - RJ, 22460-320, Brazil}
\urladdr{https://jviusos.github.io/}
\email{jviusos@math.cnrs.fr}

 % Modifications; resolution of singularities
\thanks{The first author is partially supported by CONACYT Grant No.~286445.
The second author is partially supported by MTM2016-76868-C2-2-P,
Gobierno de Arag\'on (Grupo de referencia
``\'Algebra y Geometr\'ia''), E22 17R and E22 20R, cofunded by Feder 2014--2020
``Construyendo Europa desde Arag\'on'', and by FQM-333 from Junta de Andaluc\'ia.
The third author is supported by FWO Grant No.~G079218N of Research Foundation - Flanders.
The fourth author was supported by a PNPD/CAPES grant and by a postdoctoral grant \#2016/14580-7 by
\emph{Funda\c c\~ao de Amparo \`a Pesquisa do Estado de S\~ao Paulo} (FAPESP)}

\begin{document}

\begin{abstract}
We study motivic zeta functions for $\mathds{Q}$-divisors in a $\mathds{Q}$-Gorenstein variety.
By using a toric partial resolution of singularities we reduce this study to the local case of
two normal crossing divisors where the ambient space is an abelian quotient singularity.
For the latter we provide a closed formula which is worked out directly on the
quotient singular variety. As a first application we provide a family of surface singularities
where the use of weighted blow-ups reduces the set of candidate poles drastically.
We also present an example of a quotient singularity under the action of a nonabelian group,
from which we compute some invariants of motivic nature after constructing a
$\mathds{Q}$-resolution.
\end{abstract}

\maketitle

%\tableofcontents

% ====================================================================
% Introduction
% ====================================================================
\section*{Introduction}

Roughly speaking a \emph{zeta function} is a formal power series that encodes the counting of certain mathematical objects,
traditionally of algebraic, arithmetic or geometric nature. Ideally, a zeta function preserves in some sense the algebraic,
arithmetic or geometric information of the original object. In this work we will focus on zeta functions associated with
hypersurface singularities giving rise to subtle invariants that have been studied for more than 40 years.

One of the first invariants of this type is the so-called \emph{Igusa zeta function} $\Zigusa(f; s)$ \cite{Igusa74},
it is defined as a $p$-adic parametric integral of a polynomial $f$ with coefficients in a $p$-adic field.
When $f$ is a complex polynomial, Denef and Loeser introduced
the \emph{topological zeta function} $\Ztop(f; s)$~\cite{DL92}, defined as a rational function constructed
in terms of the numerical data associated with an embedded resolution of the zero locus of~$f$. They also reinterpreted
the latter as a certain limit of the former showing in particular that $\Ztop(f; s)$ is independent of the chosen
resolution. Since the definition of $\Ztop(f; s)$ is not intrinsic, a useful and recurrent technique to study
this zeta function is the comparison of different resolutions.

The poles of both $\Zigusa(f; s)$ and $\Ztop(f; s)$ can be computed in terms of the multiplicities
and the topology of the exceptional divisors of a chosen resolution. Although each exceptional divisor in the resolution
process gives rise to a candidate pole, many of them are canceled in the calculation of the zeta function. Therefore knowing
the true poles of these zeta functions is an interesting and hard problem.
It is believed that this behavior is linked to the topology of the singularity.
More precisely, the \emph{monodromy conjecture} asserts that any pole of $\Ztop(f; s)$
provides an eigenvalue of the local monodromy action at some point of $f^{-1}(0)$.
It was proven
for some particular families of singularities, see e.g.~the references in~\cite{BoriesVeys16}, but remains widely open in general.
Typically, the strategy of the proof consists of the study of the combinatorics of the resolution to
determine the list of true poles among the candidates and comparing them with
the so-called~\emph{monodromy zeta function} obtained from the resolution via A'Campo's formula~\cite{ACampo75}.

Following Kontsevich's ideas on motivic integration~\cite{Kontsevich95}, Denef and Loeser developed a new version of
the previous zeta functions in the motivic setting over a smooth ambient space, namely the~\emph{motivic zeta function}
$\Zmot(f;s)$, see~\cite{DL98}. One of the reasons to work with $\Zmot(f;s)$ is the fact that it admits $\Zigusa(f; s)$
and $\Ztop(f; s)$ as specializations. The theory of motivic integration
over singular ambient spaces is presented in~\cite{DL99}. We refer for instance to~\cite{DL01, Craw04, Veys06, Nicaise10} for an
introduction to the theory of motivic zeta functions.

An \emph{embedded $\mathds{Q}$-resolution} is a (toric) partial resolution allowing a $V$-manifold~\cite{Steenbrink77}
with abelian quotient singularities for the final ambient space, see Section~\ref{sec:Qres}
for the details. Some evidences show that this type of resolution encodes in a compact manner the relevant information
of the singularity. For instance in~\cite{Veys97, Veys99} the use of these partial resolutions indicates clearly
that for a two-dimensional ambient space the poles for the topological and motivic zeta functions are given by
the so-called rupture divisors. Also in~\cite{Martin11, Martin16} the monodromy zeta function as well
as the Jordan blocks associated with its eigenvalues can be calculated via embedded $\mathds{Q}$-resolutions.
Other types of partial resolutions are used in~\cite{GLM97, BN16} for computing zeta functions.
Yet other approaches dealing with motivic integration and quotient singularities in the context
of Deligne-Mumford stacks are developed in~\cite{WY15,WY17,Yasuda04,Yasuda06,Yasuda14,Yasuda16}.

Inspired by the previous results and the useful description of the arc space for quotient singularities~\cite{DL02},
we study in this paper a generalization of $\Zmot(f;s)$ through $\QQ$-resolutions. In our work, much attention
is paid to the explicit calculations of $\Zmot(f;s)$ in some well selected cases,
see Sections~\ref{sc:Yomdin_Qres} and~\ref{sc:nonabelian_groups}.
In the search for a precise language to state our results,
we find it more enlightening to formulate our theory in the general framework of $\QQ$-Gorenstein varieties.
On this type of varieties the notion of motivic zeta functions associated with $\QQ$-Cartier
divisors was introduced in~\cite{Veys01}. This definition relies on a resolution of singularities
of the ambient space while our approach in this article allows one to work directly on the singular variety, see
Definition~\ref{def:motivic_zeta_function}.

For a precise description of the main results we present in this paper, some notation needs to be introduced.
Denote by $\KVarC$ the Grothendieck ring of algebraic varieties over $\CC$ and by $\Mloc=\KVarC[\LL^{-1}]$
the localization by the class of the affine line $\LL=[\AA^{1}]$. Consider $\Mhatloc$, the completion of $\Mloc$
with respect to the decreasing filtration $\{F^m\}_{m\in\ZZ}$ defined by $F^m = \langle [V] \LL^{-i} \mid
\dim V - i \leq - m \rangle$.

Let $X$ be a $\QQ$-Gorenstein algebraic variety over $\CC$ of pure dimension $n$ having at most log
terminal singularities. Denote by $\O_X$ the structural sheaf of $X$ and by $\omega_X$ the
canonical sheaf $j_{*} (\Omega_{X^{\reg}}^n)$ where $j:X^{\reg} \hookrightarrow X$ is the inclusion
of the smooth part of~$X$ and $\Omega_{X^{\reg}}^n$ is  the $n$th exterior power of the sheaf of
differentials over $X^{\reg}$. Then $\omega_X^{[r]} := j_{*} ((\Omega_{X^{\reg}}^n)^{\otimes r})$ is an invertible sheaf,
i.e.,~a locally free $\O_X$-module of rank $1$, for some $r \geq 1$. There is a measure $\mu_{\L(X)}$ on the
arc space $\L(X)$ that assigns to any $\CC[t]$-semi-algebraic subset of
$\L(X)$ a value in $\Mhatloc [\LL^{1/r}]$. In this way one can construct integrals of exponential functions $\LL^{-\alpha}$
where $\alpha: A \to \frac{1}{r} \ZZ \cup \{\infty\}$ is a $\CC[t]$-simple function bounded from below,
see Section~\ref{subsec:arcs_and_motivic} for the details.
In particular one can show, using the change of variables formula for a resolution of $X$,
that $\LL^{-\frac{1}{r}\ord_t \omega_X^{[r]}}$ is integrable on $X$. For every measurable subset
$A$ of $\L(X)$ one defines its $\QQ$-\emph{Gorenstein measure} by
\begin{equation}\label{def:Gor-measure}
\mu^{\Gor}_{\L(X)}(A) = \int_A \LL^{-\frac{1}{r}\ord_t \omega_X^{[r]}}
\dd \mu_{\L(X)} \in \Mhatloc[\LL^{1/r}].
\end{equation}
This notion appears for instance in \cite[Section~7.3.4]{ChNS18}. Note that every orbifold is $\QQ$-Gorenstein. However, the definition of orbifold measure given
in~\cite[Section~3.7]{DL02} is different from this one, since ours is intrinsic to $X$ while
the one in loc.~cit.~depends on the order of the Jacobian of a projection, see Example~\ref{ex:Gor-measure}.

Consider two effective $\QQ$-Cartier divisors $D_1$ and $D_2$ in $X$; take $r$ such that $r D_1$ and $r D_2$ are Cartier
and such that $\omega_X^{[r]}$ is invertible. There is a natural way to define
$\ord_tD_i:\L(X)\to\frac{1}{r}\ZZ\cup\{\infty\}$, namely $\ord_t D_i=\frac{1}{r} \ord_t (rD_i)$, $i=1,2$.
We introduce the notion of motivic zeta function in this context,
using the Gorenstein measure $\mu^{\Gor}_{\L(X)}$ defined in~\eqref{def:Gor-measure}, which encodes better
the structure of $X$. As in~\cite{Veys01, NemethiVeys12}, we also consider
the zeta function $\Zmot(D_1,D_2; s)$ associated to a pair $(D_1,D_2)$, where $D_2$ may be thought of as
a divisor associated with some $n$-form on $X$.

\begin{defi_intro}\label{def:motivic_zeta_function}
Let $W$ be a subvariety of $X$ and consider $\L(X)_W = \tau_0^{-1}(W)$, $\tau_0$ being the truncation map
$\tau_0: \L(X) \to X$, and set $\L(X)_W^{\reg}=\L(X)_W\setminus\L(X_{\sing})$.
The ($\QQ$-\emph{Gorenstein}) \emph{motivic zeta function} of the pair $(D_1,D_2)$
with respect to $W$ is
\[
\Zmot_{,W}(D_1,D_2; s) = \int_{\L(X)_{W}^{\reg}} \LL^{-(\ord_t D_1\cdot  s + \ord_tD_2)} \dd \mu^{\Gor}_{\L(X)},
\]
whenever the right-hand side converges in $\Mhatloc [\LL^{1/r}]\llbracket\LL^{- s/r}\rrbracket$.
When $W$ is just a point $P \in X$, the zeta function is simply called the \emph{local motivic zeta function}
at $P$ and it is denoted by $\Zmot_{,P}(D_1,D_2; s)$.
\end{defi_intro}

If $D_1$ and $D_2 + \Div(\omega_X)$ are Cartier and the right-hand side above converges, then $\Zmot_{,W}(D_1,D_2; s)$
is an element of $\Mhatloc\llbracket\LL^{- s}\rrbracket$.
Note that the previous definition of the motivic zeta function as a motivic integral is a generalization
of the classical one given in~\cite{DL98,DL99}, or in~\cite{NemethiVeys12, CauwbergsVeys17}, where different
$D_2$ associated with $n$-differential forms over a smooth $X$ are considered.
It is worth noticing that our computations are performed in $\Mhatloc$, as opposed to~\cite{DL02} where the
authors work in a quotient ring of $\Mhatloc$ assuming that the class of a quotient of a vector space $V$
by a finite group acting linearly should be that of~$V$.

The change of variables formula is one of the main tools for computing motivic integrals.
In the $\QQ$-Gorenstein case we present a formula that is stated in Theorem~\ref{thm:intro_orbs_chg_vars} below.
An important fact is that in the singular case the change of variables requires the understanding
of the order of the Jacobian of the corresponding morphism, which is often a hard task.
Note that computing the order of the Jacobian involves an explicit description of the sheaf of K\"ahler differential
$n$-forms. Using the $\QQ$-Gorenstein measure we were able to avoid the
computation of $\ord_t\Jac$ by computing the order of the relative canonical divisor, which is useful
in the applications we present.

\begin{theoremintro}\label{thm:intro_orbs_chg_vars}
Let $X$ and $Y$ be two $\QQ$-Gorenstein varieties of pure dimension $n$. Consider a proper birational map $\pi: Y \to X$ and a subvariety $W \subset X$. Then
\[
\Zmot_{,W}(D_1,D_2;  s)=\Zmot_{,\pi^{-1}W}(\pi^*D_1,\pi^*D_2 + K_{\pi};  s),
\]
where $K_\pi$ denotes the relative canonical divisor associated with $\pi$.
\end{theoremintro}

Consider now $U = \CC^n/G$, where $G \subset \GL_n(\CC)$ is a finite abelian group of order $d$. The structural
sheaf of $U$ is identified with $\O_{\CC^n}^G$ given by the $G$-invariant elements of $\O_{\CC^n}$.
We will denote by $x_1,\dots,x_n$ the coordinates on both $\CC^n$ and $U$.
Note that
the notion of $\QQ$-Weil divisor and $\QQ$-Cartier divisor coincide in this setting, besides $r$ can always be
chosen to be $d$. Let us fix a primitive $d$th root of unity $\zeta_d$. There exists a basis of $\CC^n$
such that any $\gamma\in G$ is a diagonal matrix of the form
\begin{equation}\label{eq:our_diag_matrix}
\text{Diag} \left( \zeta_d^{\varepsilon_{\gamma,1}}, \ldots, \zeta_d^{\varepsilon_{\gamma,n}} \right),
\end{equation}
with $0\leq \varepsilon_{\gamma,i}\leq d-1$, $i=1\ldots,n$. For any tuple $\kbf=(k_1,\ldots,k_n) \in \QQ^n$,
define the map $\varpi_\kbf: G \to \QQ$ given by
\begin{equation}\label{eq:def-weight-function}
\varpi_\kbf(\gamma) =
\frac{1}{d}\sum_{i=1}^{n} k_i\varepsilon_{\gamma,i}.
\end{equation}
Our second result expresses the $\QQ$-normal crossing situation for an abelian quotient space $U$
in terms of inner data of the underlying group and the multiplicities. %
Recall that a group in $\GL_n(\CC)$ is called \emph{small} if it does not contain rotations around hyperplanes other than the identity, see Definition~\ref{def:small}.

\begin{theoremintro}\label{thm:intro_main1}
Let $D_1$ and $D_2$ be divisors on $U$ given by $x_1^{N_1} \cdots x_n^{N_n}$ and
$x_1^{\nu_1-1} \cdots x_n^{\nu_n-1}$ with $N_i,\nu_i-1 \in \frac{1}{d} \ZZ_{\geq 0}$ for $i=1\ldots,n$.
Assume that $G$ is a small group acting diagonally as in (\ref{eq:our_diag_matrix}) and
denote $\Nbf = (N_1,\ldots,N_n)$ and $\nubf = (\nu_1,\ldots,\nu_n)$. Then
\[
\Zmot_{,0}(D_1,D_2; s) = S_G(\Nbf,\nubf;s)
\LL^{-n} \prod_{i=1}^n \frac{(\LL-1) \LL^{-(N_i s+\nu_i)}}{1-\LL^{-(N_i s+\nu_i)}},
\]
where $$S_G(\Nbf,\nubf;s) = \sum_{\gamma \in G} \LL^{\varpi_\Nbf(\gamma)\cdot  s + \varpi_\nubf(\gamma)}$$
and $\varpi_{\Nbf}$, $\varpi_{\nubf}$ are defined in~\eqref{eq:def-weight-function} above.
\end{theoremintro}

This result can be reformulated for groups that are not small as well, see Remark~\ref{rk:comparison-DL}.
We refer to~\cite{BN16} for a related statement in the context of log geometry.
Observe that the term \[\LL^{-n} \prod_{i=1}^n \frac{(\LL-1) \LL^{-(N_i s+\nu_i)}}{1-\LL^{-(N_i s+\nu_i)}}\]
corresponds to the local motivic
zeta function of $x_1^{N_1} \cdots x_n^{N_n}$ and $x_1^{\nu_1-1} \cdots x_n^{\nu_n-1}$ over $\CC^n$,
while $S_G(\Nbf,\nubf;s)$ is related to the standard weight function of $G$ twisted by the multiplicities
$N_i$ and $\nu_i$. The standard weight function is also called \emph{age} in \cite{Reid80, Kollar13}.
Moreover, if $D_1$ and $D_2 + \Div (\omega_U)$ are Cartier, then the sum $S_G(\Nbf,\nubf;s)$ belongs
to $\Mhatloc[\LL^{s}]$.
When $N_i=\nu_i-1=0$ for all $i=1,\ldots,n$, the result above
extends~\cite[Theorem 3.6]{DL02} for abelian quotient singularities, see Remark~\ref{rk:comparison-DL}.

Combining Theorem~\ref{thm:intro_orbs_chg_vars} and Theorem~\ref{thm:intro_main1} above, one obtains
an explicit formula for the motivic zeta function in terms of an embedded $\QQ$-resolution
$\pi: Y\to X$ of $D_1+D_2$. This can be understood as a partial resolution allowing $Y$ to
contain abelian quotient singularities and $\pi^{*} (D_1 + D_2)$ to have $\QQ$-normal crossings
in the sense of Steenbrink~\cite{Steenbrink77}. In this situation, there is a natural finite
stratification $Y = \bigsqcup_{k \geq 0} Y_{k}$ such that each stratum is characterized by the
following condition. Given $q \in Y_k$ there exist a finite abelian group $G_k \subset \GL_n(\CC)$
acting diagonally as in~\eqref{eq:our_diag_matrix},
$\Nbf_k = (N_{1,k},\ldots,N_{n,k}) \in \QQ^n_{\geq 0}$,
and $\nubf_k = (\nu_{1,k},\ldots,\nu_{n,k}) \in \QQ^n_{\geq 0}$ such that $Y$ is locally isomorphic around $q$
to $U_k = \CC^n/G_k$ and the local equations of $\pi^{*} D_1$ and $\pi^{*} D_2 + K_{\pi}$ are given by
$x_1^{N_{1,k}} \cdots x_n^{N_{n,k}}$ and $x_1^{\nu_{1,k}-1} \cdots x_n^{\nu_{n,k}-1}$.
Moreover, the data $G_k$, $\Nbf_k$, $\nubf_k$ do not depend on the chosen $q \in Y_k$ but
only on the stratum $Y_k$. For the next result recall that $rD_1$ and $rD_2$ are Cartier and
$\omega_X^{[r]}$ is invertible.

\begin{theoremintro}\label{thm:intro_main2}
Using the previous notation one has
\[
\Zmot_{,W}(D_1,D_2;s) = \LL^{-n} \sum_{k \geq 0}
\left[ Y_k \cap \pi^{-1}(W) \right] S_{G_k}(\Nbf_k,\nubf_k;s)
\prod_{i = 1}^n \frac{(\LL-1) \LL^{-(N_{i,k} s+\nu_{i,k})}}{1-\LL^{-(N_{i,k} s+\nu_{i,k})}}.
\]
Moreover, the zeta function $\Zmot_{,W}(D_1,D_2;s)$ is a rational function and it belongs
to the subring of $\Mhatloc [\LL^{1/r}]\llbracket\LL^{- s/r}\rrbracket$ generated by $\Mloc$, $\LL^{1/r}$,
$\LL^{-s/r}$, and $\set{ \frac{\LL-1}{1-\LL^{-(a s+b)}} }_{a,b-1\in \frac{1}{r}\ZZ}$.
\end{theoremintro}

A version of the previous result for the topological and Igusa zeta functions was already stated in~\cite{Veys97} for
$X=\CC^2$, $D_1$ an effective Cartier divisor, and $D_2=0$. In loc.~cit., the sum $S_{G_k}$
appears as a certain non-symmetric deformation of the intersection matrix of the minimal resolution of a
Hirzebruch-Jung singularity. Note that our formula for $S_{G_k}$ is very explicit and depends
only on the local action of $G_k$ on each stratum. This gives a conceptual explanation
for the appearance of the determinant $\mathcal{D}_r$ in \cite[Definition 5.5]{Veys97}, which is not well understood,
see Remark~\ref{rk:comparison-Veys} and Example~\ref{ex:comparison-Veys}.

When $\LL\to 1$, the sum $S_G(\Nbf,\nubf;s)\to d=|G|$. Hence, specializing to the Euler characteristic gives
rise to a formula for the topological zeta function in terms of a $\QQ$-resolution, where the contribution of each
stratum is multiplied by the order of the corresponding group, see Corollaries~\ref{cor:top1} and~\ref{cor:top2}.
Also, as is well known by specialists, the $p$-adic Igusa zeta function as well as the Hodge
and  arc-Euler characteristic zeta function are obtained by specialization of the motivic zeta function.

As a first application, we show in Section~\ref{sc:Yomdin_Qres} how our techniques reduce the number of candidate
poles for the motivic zeta function by providing an explicit calculation for a family of Yomdin surface singularities
using $\QQ$-resolutions.
Now, given a normal crossing divisor in a quotient singularity by a nonabelian group, Batyrev \cite{Batyrev99,Batyrev00}
developed a method to arrive at a birationally equivalent model with only abelian singularities.
Following these ideas we present in Section~\ref{sc:nonabelian_groups} an example of a tetrahedral singularity
and compute some motivic invariants.

The paper is organized as follows. In Section~\ref{sc:settings}, we give the necessary definitions
and notations for the concepts to be dealt with in this paper such as arc spaces, motivic integration,
$\QQ$-Gorenstein measure, and embedded $\QQ$-resolutions.
In Section~\ref{sc:def_zeta_function}, we introduce the motivic zeta function in the context of
$\QQ$-Gorenstein varieties, see Definition~\ref{def:motivic_zeta_function}, and we prove the change of variables
formula described in Theorem~\ref{thm:intro_orbs_chg_vars}. Section~\ref{sc:abelian_quotient_sing}
is devoted to the quotient case under a finite abelian group and the formula in terms of an embedded $\QQ$-resolution,
see Theorem~\ref{thm:intro_main1} and Theorem~\ref{thm:intro_main2}. Finally, in Sections~\ref{sc:Yomdin_Qres}
and~\ref{sc:nonabelian_groups} we present some examples and applications of our techniques.

\subsection*{Acknowledgments} We thank Enrique Artal and Jos\'e Ignacio Cogolludo for fruitful discussions
during the preparation of this manuscript. Some of the ideas behind this work arose at the Third International
Workshop on Zeta Functions in Algebra and Geometry held in Guanajuato.
The second author wants to thank the Fulbright Program (within the Jos\'e Castillejo grant by Ministerio de
Educaci\'on, Cultura y Deporte) for its financial support while writing this paper. He also thanks
the University of Illinois at Chicago, especially Lawrence Ein for his warm welcome and
support in hosting him.

% We would like to thank the referee for carefully reading the manuscript and pointing out useful suggestions.

% ====================================================================
% Settings on motivic integration over quotient spaces
% ====================================================================
\bigskip
\section{Settings and preliminaries}\label{sc:settings}

\subsection{Preliminaries on arc spaces and motivic integration}\mbox{}\label{subsec:arcs_and_motivic}

The theory of arc spaces and motivic integration has an interesting history, that the reader may consult
for instance in \cite{Veys06},
or more recently in~\cite{ChNS18}.

Denote by $\KVarC$ the Grothendieck ring of algebraic varieties over $\CC$. This is the free abelian group
generated by the symbols $[V]$, where $V$ is a variety, subject to the relations $[V]=[V']$ if $V \cong V'$
and $[V] = [V \setminus V'] + [V']$ if $V'$ is closed in $V$. Its ring structure is given by
$[V] \cdot [V'] = [V \times V']$. Set $\LL=[\AA^{1}]$ and denote $\Mloc=\KVarC[\LL^{-1}]$, the localization
by the class of the affine line. Consider the decreasing filtration $\{F^m\}_{m\in\ZZ}$ on $\Mloc$,
where $F^m$ is the subgroup of $\Mloc$ generated by $\{[V] \LL^{-i} \mid \dim V - i \leq - m\}$
and denote by $\Mhatloc$ the completion of $\Mloc$ with respect to this filtration.

Let $X$ be an algebraic variety over $\CC$, not necessarily smooth. For each natural number $m$ we consider
the \textit{space $\L_m(X)$ of $m$-jets on $X$}. This is a complex algebraic variety, whose $\CC$-rational
points are the $\CC[t]/\langle t^{m+1}\rangle$-rational points of $X$. The projective limit of these
algebraic varieties $\L_m(X)$ is the \textit{arc space $\L(X)$ of $X$}, which is a reduced separated
scheme over $\CC$ (in general not of finite type). Then the $\CC$-rational points of $\L(X)$ are
precisely the $\CC\llb t\rrb$-rational points of $X$.

For any $m$, and for $n \geq m$, we have natural morphisms
\[
\tau_m: \L(X)\to \L_m(X)\ \text{ and }\ \tau_m^{n}: \L_n(X)\to \L_m(X),
\]
obtained by truncation. Note that $\L_0(X)=X$. For any arc $\varphi$ on $X$, we call $\tau_0(\varphi)$
the origin of $\varphi$ and if $W$ is a subvariety of $X$, we set $\L(X)_W=\tau_0^{-1}(W)$.

A minor variation of the space of $m$-jets and the space of arcs of $X$ is required in the study of
$\QQ$-Gorenstein varieties. Given an integer $d\geq1$, the space $\L_m^{1/d}(X)$ of
\textit{ramified $m$-jets on $X$} is a complex algebraic variety, whose $\CC$-rational points,
are the $\CC[t^{1/d}]/\langle t^{(m+1)/d}\rangle$-rational points of $X$. The projective limit
of these $\L_m^{1/d}(X)$ is the \textit{ramified arc space $\L^{1/d}(X)$ of $X$}. Note that
$\L^{1/d}(X)$ is isomorphic to $\L(X)$. We shall still use $\tau_m$ to denote the canonical
morphism $\tau_m: \L^{1/d}(X)\to \L_m^{1/d}(X)$ and for $W$ a subvariety of $X$,
we set $\L^{1/d}(X)_W=\tau_0^{-1}(W)$.

Let $X$ and $Y$ be two complex varieties. A function $\phi: \L(Y)\to \L(X)$, is called a
\textit{$\CC[t]$-morphism} if it is induced by a morphism of $\CC[t]$-schemes
$Y\otimes_\CC\CC[t]\longrightarrow X\otimes_\CC\CC[t]$.

Our next target is the construction of a motivic measure over the completion $\Mhatloc$, introduced before. Recall that a subset
of a variety $X$ is called \textit{constructible} when it is a finite union of (locally closed)
subvarieties. It is well known that $\tau_m(\L(X))$ is a constructible subset of the algebraic
variety $\L_m(X)$. Furthermore, when $X$ is smooth, $\tau_m$ is surjective, and $\tau_m^n$ is
a locally trivial fibration with fiber $\AA^{(n-m) \dim X}$.

Let $X$ be a complex algebraic variety of pure dimension $n$. A subset $C$ of $\L(X)$ is called
\textit{constructible} if $C=\tau_m^{-1}(B)$ with $B$ a constructible subset of $\L_m(X)$ for
some $m>0$. A subset $C$ of $\L(X)$ is called \textit{stable} if it is constructible and
$C\cap \L(X_{\sing})=\emptyset$. When $C\subset \L(X)$ is stable,  the elements
$[\tau_m(C)]\LL^{-(m+1)n}$  in $\Mloc$ stabilize for $m$ big enough, and
\[
\tilde{\mu}(C) = \lim\limits_{m\to\infty}[\tau_m(C)]\LL^{-(m+1)n}\in\Mloc
\]
is called the \textit{naive motivic measure of} $C$. This claim follows, if $X$ is smooth,
from the fact that $\tau_m^{m+1}$ are locally trivial fibrations with fiber $\AA^{n}$.
In the case $X$ is singular, the claim follows from \cite[Lemma 4.1]{DL99}. When $C$ is not
stable, $[\tau_m(C)]\LL^{-(m+1)n}$ will not always stabilize. However, it can be proven that
the limit \[{\mu}(C) = \lim\limits_{m\to\infty}[\tau_m(C)]\LL^{-(m+1)n}\] exists in the completed
Grothendieck group $\Mhatloc$. The element $\mu(C)$ of $\Mhatloc$ is called
\textit{the motivic measure of} $C$. This yields a $\sigma$-additive measure $\mu$ on the Boolean
algebra of constructible subsets of $\L(X)$.

A special family of measurable sets consists of the semi-algebraic subsets of $\L(X)$, which can be defined from the affine case by using charts.
Roughly speaking, a \textit{semi-algebraic subset of} $\L(\AA^n)$ is a finite Boolean combination of subsets defined by conditions involving (in)equalities between orders of $\CC[t]$-polynomial equations of elements in $\CC((t))$ or their lowest degree coefficients. We refer to~\cite[Section~2]{DL99} for a formal definition.
One can prove that for every semi-algebraic subset $C$ of
$\L(X)$, the measure $\mu(C)$ is precisely
$\lim\limits_{m\to\infty}[\tau_m(C)]\LL^{-(m+1)n}\in\Mhatloc$.

Let $A$ be a measurable subset of $\L(X)$; a function $\alpha: A\rightarrow \ZZ\cup\{\infty\}$ is called
\textit{simple} if all of its fibres are measurable. For such $\alpha$ we say $\LL^{-\alpha}$ is \textit{integrable} if
the series \[\int_{A}\LL^{-\alpha}\dd\mu = \sum_{i \in \ZZ}\mu(A\cap\alpha^{-1}(i))\LL^{-i}  \] is
convergent in $\Mhatloc$. Note that this will always be the case when $\alpha$ is bounded from below.

Before stating one of the main results of the theory of motivic integration, namely the change of
variables formula, some definitions are in order. When $X$ and $Y$ are smooth varieties of pure
dimension and $h:\L(Y)\rightarrow\L(X)$ a $\CC[t]$-morphism, $\ord_t\Jac_{h}$ is simply defined as the order of the local generator of the ordinary Jacobian
determinant with respect to local coordinates on $X$ and $Y$. In the singular case the definition
is more involved; the following is a slight generalization of~\cite[Section~1.14]{DL02}.
We set $\Omega^n_X$ for the $n$th exterior power of the sheaf of differential forms on a variety $X$.

\begin{defi}\label{def:order-jacobian}
Let $X$ and $Y$ be complex varieties of pure
dimension $n$ and let $h:\L(Y)\rightarrow\L(X)$ be a $\CC[t]$-morphism.
Take $\psi\in\L(Y)\setminus\L(Y_{\sing})$.
We consider $\psi^\ast(\Omega_Y^n)$ as a $\CC\llb t\rrb$-module and denote by $L_Y$ its image in the
$\CC((t))$-vector space $V=\psi^\ast(\Omega_Y^n)\otimes_{\CC\llb t\rrb}\CC((t))$. Hence $L_Y$ is a lattice of rank 1 in $V$.

\begin{itemize}
	\item[(1)] Consider also the image $L_X$ of the module $\psi^\ast h^\ast(\Omega_{X\otimes\CC[t]}^n)$ in $V$. If $L_X$
	is nonzero, then $L_X=t^e L_Y$ for some $e\in\NN$ and we set $\ord_t\Jac_{h}(\psi)=e$. When $L_X=0$, we put
	$\ord_t\Jac_{h}(\psi)=\infty$.
	
	\item[(2)] Let $\omega$ be an invertible $\O_Y$-subsheaf of $\Omega^n_Y\otimes_\CC \CC(Y)$.
	Denote by $\Lambda_Y$ the
	image of $\psi^\ast(\omega)$ in $V$.  If $\Lambda_Y$ is nonzero, then
	$\Lambda_Y=t^e L_Y$ for some $e\in\ZZ$ and we set $\ord_t\omega(\psi)=e$. When $\Lambda_Y=0$,
	we put $\ord_t\omega(\psi)=\infty$.
	
	\item[~(2')] More generally, fix a positive integer $r$ and let $\omega'$ be an invertible $\O_Y$-subsheaf of $(\Omega^n_Y)^{\otimes r}\otimes_\CC \CC(Y)$.
	Then we denote by $\Lambda_Y$ the
	image of $\psi^\ast(\omega')$ in $V^{\otimes r}$ and we define analogously $\ord_t\omega'(\psi)$ as the number $e\in\ZZ \cup \{\infty\}$ such that
	$\Lambda_Y=t^e L_Y^{\otimes r}$.
\end{itemize}
\end{defi}

By \cite[Lemma~1.15]{DL99} the maps $\ord_t\Jac_{h}$, $\ord_t\omega$ and $\ord_t\omega'$ are simple.

\begin{theorem}[{\cite[Theorem~1.16]{DL02}}, Change of variables formula]\label{thm:chg_var_DL}
	Let $X$ and $Y$ be complex algebraic varieties of pure dimension $n$, and let $h:\L(Y)\rightarrow\L(X)$
	be a $\CC[t]$-morphism. Let $A$ and $B$ be two semi-algebraic sets in $\L(X)$ and $\L(Y)$, respectively,
	such that $h$ induces a bijection between $B$ and $A$. Then, for any simple function
	$\alpha:A\to\ZZ\cup\{\infty\}$ such that $\LL^{-\alpha}$ is integrable on $A$, we have
	\[
		\int_A \LL^{-\alpha}\dd\mu=\int_B \LL^{-\alpha\circ h-\ord_t\Jac_h(\varphi)}\dd\mu.
	\]
\end{theorem}

\begin{remark}\label{rmk:order-divisor}\mbox{}
	An important example of an integrable function is induced by an effective Cartier divisor $D$ on $X$.
	We define $\ord_tD:\L(X)\to\NN\cup\{\infty\}: \varphi\mapsto \ord_t f_D(\varphi)$, where $f_D$ is a
	local equation of $D$ in a neighborhood of the origin $\tau_0(\varphi)$ of $\varphi$. Note that
	$\ord_t D(\varphi)=\infty$ if and only if $\varphi\in\L(D^{\red})$ and $\ord_t D(\varphi)=0$ if and
	only if $\tau_0(\varphi)\notin D^{\red}$. It is not difficult to prove that $\LL^{-\ord_tD}$ is
	integrable on $\L(X)$. When $X$ and $Y$ are both smooth and $h$ is induced by a proper birational
	morphism $h: Y\to X$, the change of variables formula can be written as
	\begin{equation}\label{eq:chg_var_div}
	\int_A \LL^{-\ord_tD}\dd\mu=\int_B \LL^{-\ord_t(h^\ast D+K_{h})}\dd\mu.
	\end{equation}
	Here $h^\ast D$ denotes the pullback of $D$ and $K_h$ denotes the relative canonical divisor.
	
Consider an effective $\QQ$-Cartier divisor $D$ in $X$, and say $r D$ is Cartier. In this case we define
\[
\ord_t D=\frac{1}{r} \ord_t (rD).
\]
\end{remark}

\subsection{Arc spaces on quotient singularities}\label{sc:arc-spaces}\mbox{}

Let $d\geq1$ be an integer and fix $\zeta$ a primitive $d$th root of unity in $\CC$. Let $G$ be a finite
subgroup of $\GL_n(\CC)$ of order $d$. For any $\gamma\in G$, there exists a basis $\{b_i^\gamma\}_{1\leq i\leq n}$ of
$\CC^n$ such that $\gamma$ is a diagonal matrix of the form
\begin{equation}\label{eq:diag_matrix}
\left(\begin{array}{ccc}
\zeta^{e_{\gamma,1}} & & \\ & \ddots & \\ & & \zeta^{e_{\gamma,n}}
\end{array}\right)
\end{equation}
with $1\leq e_{\gamma,i}\leq d$ for $i=1,\ldots, n$.

Furthermore, given an integer vector $\kbf=(k_1,\ldots,k_n)$, we define the map
\[
\begin{array}{cccl}
w_\kbf: & G & \longrightarrow & \frac{1}{d}\ZZ \\
& \gamma & \mapsto & w_\kbf(\gamma) = \displaystyle \frac{1}{d} \sum_{i=1}^{n} k_i e_{\gamma,i}.
\end{array}
\]
The \textit{weight} of  $\gamma\in G$ is the value $w(\gamma) = w_{(1,\ldots,1)}(\gamma)$. Note that in general
$w_\kbf(\gamma)$ is a rational number.

Analogously, define the map
\[
\begin{array}{cccl}
\varpi_\kbf: & G & \longrightarrow & \frac{1}{d}\ZZ\\
& \gamma & \mapsto & \varpi_\kbf(\gamma) = \displaystyle \frac{1}{d} \sum_{i=1}^{n} k_i\varepsilon_{\gamma,i}, \\
\end{array}
\]
where the numbers $\varepsilon_{\gamma,1},\ldots,\varepsilon_{\gamma,n}$ are taken as in (\ref{eq:diag_matrix})
but now verifying $0\leq \varepsilon_{\gamma,i}\leq d-1$, for $i=1,\ldots, n$. Note that the relation
$e_{\gamma,i}+\varepsilon_{\gamma^{-1},i}=d$ holds for any $i=1,\ldots,n$, and that we have
\begin{equation}\label{eq:weights}
w_\kbf(\gamma)+\varpi_\kbf(\gamma^{-1})=\sum_{i=1}^{n}k_i,
\end{equation}
for any $\gamma\in G$.

Now, we let $G$ act on $\CC^n$ and consider the projection morphism $\rho: \CC^n\rightarrow U = \CC^n/G$ as
a morphism of complex varieties. In~\cite{DL02} the structure of the arcs over  $U$ centered
at the origin is studied. We summarize in this section some of their results.

Let $\tilde{\Delta}$ be the closed subvariety of $\CC^n$ consisting of the closed points having a nontrivial
stabilizer and let $\Delta$ be its image under $\rho$ in $U$. We denote by $\L(U)^{\reg}$
(resp.~$\L^{1/d}(\CC^n)^{\reg}$) the set $\L(U)\setminus\L(\Delta)$
(resp.~$\L^{1/d}(\CC^n)\setminus\L^{1/d}(\tilde{\Delta})$). We define similarly
$\L_W(U)^{\reg}$ and $\L_W^{1/d}(\CC^n)^{\reg}$, for a subvariety $W$ of $U$ (resp.~$\CC^n$).

Note that an arc $\varphi\in\L(U)_0^{\reg}$ can be lifted to an arc $\tilde{\varphi}\in\L^{1/d}(\CC^n)_0^{\reg}$
in such a way that there is a unique element $\gamma\in G$ verifying
\begin{equation}\label{eq:conj-arcs}
\tilde{\varphi}(\zeta \, t^{1/d})=\gamma \tilde{\varphi}(t^{1/d}).
\end{equation}
When $\tilde{\varphi}$ is replaced by another arc in $\L^{1/d}(\CC^n)_0^{\reg}$, the element $\gamma$ in \eqref{eq:conj-arcs}
will be replaced by a conjugate. Denote by $\L(U)_{0,\gamma}^{\reg}$ the set of arcs in $\L(U)_0^{\reg}$ such that
there exists $\tilde{\varphi}$ satisfying \eqref{eq:conj-arcs}; then
$\L(U)_{0,\gamma}^{\reg}=\L(U)_{0,\gamma^\prime}^{\reg}$ if and only if  $\gamma$ and $\gamma^\prime$ belong to
the same conjugacy class of $G$. Summarizing, we have
\begin{equation}\label{eq:decomposition-arcs}
\L(U)^{\reg}_0=\bigsqcup_{\gamma \in \Conj(G)}\L(U)^{\reg}_{0,\gamma},
\end{equation}
where $\Conj(G)$ denotes the conjugacy classes of $G$. Moreover, for a fixed $\gamma\in G$, we have that an arc
$\tilde{\varphi}\in\L^{1/d}(\CC^n)^{\reg}$ projects to an arc $\varphi$ in $\L(U)_{0,\gamma}^{\reg}$ if and only if $\tilde{\varphi}$ it is
in the $G$-orbit of some arc $\tilde{\vartheta}$ in $\L^{1/d}(\CC^n)^{\reg}$ of the form
\begin{equation}\label{eq:lifted-arcs}
\tilde{\vartheta}(t^{1/d})=(t^{e_{\gamma,1}/d}\varphi_1(t),\ldots,t^{e_{\gamma,n}/d}\varphi_n(t)).
\end{equation}

In fact, when we fix an element $\gamma\in G$ and a base  $\{b_i^\gamma\}$ of $\CC^n$ such that $\gamma$ is diagonal,
\eqref{eq:lifted-arcs} gives rise to the following $\CC[t]$-morphism of varieties:
\begin{equation}\label{eq:lambda}
	\begin{array}{cccc}
		\lambda_\gamma: & \L(U) & \longrightarrow & \L(U)\\
				& (\varphi_1(t),\ldots,\varphi_n(t))
				& \mapsto & (t^{e_{\gamma,1}/d}\varphi_1(t),\ldots,t^{e_{\gamma,n}/d}\varphi_n(t)).
	\end{array}
\end{equation}
The map  $\lambda_\gamma$ gives  a bijection between $\L(U)^{\reg}_{0,\gamma}$ and
$\left(\L(\CC^n)/G_\gamma\right)\cap\lambda_\gamma^{-1}\left(\L(U)^{\reg}\right)$, where $G_\gamma$ denotes
the centralizer of $\gamma\in G$.
In particular,  $\L(U)^{\reg}_{0,\gamma}$ is a semi-algebraic subset of $\L(U)$ and since
$\lambda_\gamma^{-1}\L(\Delta)$ is a subvariety of dimension less than $n$, the semi-algebraic sets
$\L(\CC^n)/G_\gamma$ and $\left(\L(\CC^n)/G_\gamma\right)\cap\lambda_\gamma^{-1}\left(\L(U)^{\reg}\right)$
are equal up to sets of measure zero. If $G$ is abelian, then $G_\gamma=G$ for any $\gamma\in G$,
and everything is reduced to the study of $\L(\CC^n)/G$.

\subsection{Motivic $\QQ$-Gorenstein measure}\mbox{}\label{subsec:mot_orb_measure}

Let $X$ be a $\QQ$-Gorenstein algebraic variety over $\CC$ of pure dimension $n$ having at most log
terminal singularities. Denote by $\O_X$ the structural sheaf of $X$ and by $\omega_X$ the
canonical sheaf $j_{*} (\Omega_{X^{\reg}}^n)$ where $j:X^{\reg} \hookrightarrow X$ is the inclusion
of the smooth part of~$X$ and $\Omega_{X^{\reg}}^n$ is  the $n$th exterior power of the sheaf of
differentials over $X^{\reg}$. Then $\omega_X^{[r]} := j_{*} ((\Omega_{X^{\reg}}^n)^{\otimes r})$ is an invertible sheaf,
i.e.,~a locally free $\O_X$-module of rank $1$, for some $r \geq 1$.

Now we consider the ring $\Mhatloc [\LL^{1/r}]$ constructed by extending  $\Mhatloc$ by $[\LL^{1/r}]$
and then completing with respect to the filtration generated by
\[
\left\{[V] \LL^{-i/r}\ \left|\ \dim V - \frac{i}{r} \leq - \frac{m}{r}\right.\right\}.
\]
This way one can show that if $A$ is a semi-algebraic subset of $\L(X)$ and $\alpha: A \to \frac{1}{r} \ZZ \cup
\{\infty\}$ is a simple function, then $\LL^{-\alpha}$ is integrable in $\Mhatloc [\LL^{1/r}]$. In particular,
by using a resolution of singularities for $X$ and Theorem \ref{thm:chg_var_DL}, one has that
$\LL^{-\frac{1}{r}\ord_t \omega_X^{[r]}}$ is integrable on $A$. Thus the following definition
of a $\QQ$-\emph{Gorenstein measure} for a measurable subset $A$ of $\L(X)$ makes sense.
\begin{defi}[Motivic $\QQ$-Gorenstein measure]\label{def:Gor-measure2}
	Using the notation above, one defines
	\[
	\mu^{\Gor}(A) = \int_A \LL^{-\frac{1}{r}\ord_t \omega_X^{[r]}} \dd \mu_{\L(X)} \in \Mhatloc[\LL^{1/r}].
	\]
\end{defi}
Note that this notion is called  \emph{Gorenstein measure} in \cite[Section~7.3.4]{ChNS18}. As a word of caution we present an example to show that this definition differs from the \emph{orbifold measure} $\mu^{\orb}$ given in~\cite[Section~3.7]{DL02} when the variety is a quotient singularity by the action of a finite linear group. For any $U=\CC^n/G$ with $G\subset\GL_n(\CC)$ of finite order, it is proven in \cite{DL02} that
\begin{equation}\label{eqn:mu_orb}
	\mu^{\orb}(\L(U)_0) = \sum_{\gamma\in\Conj(G)} \LL^{-w(\gamma)}.
\end{equation}

\begin{ex}\label{ex:Gor-measure}
Let $U_i = \CC^2/G_i$, $i=1,2$, be the quotient spaces given by $G_1 = \{ \id_2, -\id_2 \}$  and
$G_2 = \big\{ \id_2, \begin{psmallmatrix} i & 0 \\ 0 & -1 \end{psmallmatrix},
\begin{psmallmatrix} -1 & 0 \\ 0 & 1 \end{psmallmatrix}, \begin{psmallmatrix} -i & 0 \\ 0 & -1 \end{psmallmatrix}  \big\}$, respectively. The elements of $G_1$ have weight $2$ and $1$ and those of $G_2$ have weight $2$, $3/4$, $3/2$ and $5/4$, respectively. According to (\ref{eqn:mu_orb}), we have
$$
  \mu^{\orb}(\L(U_1)_0) = \LL^{-2} + \LL^{-1} = \LL^{-2} (1 + \LL)
$$
and
$$
\mu^{\orb}(\L(U_2)_0) = \LL^{-2} + \LL^{-3/4} + \LL^{-3/2} + \LL^{-5/4}
= \LL^{-2} ( 1 + \LL^{1/2} ) ( 1 + \LL^{3/4} ).
$$
However, the map $U_2 \to U_1$ defined by $[(x,y)] \mapsto [(x^2,y)]$ provides an isomorphism
between $U_2$ and $U_1$. This shows that the orbifold measure depends not only on the algebraic
structure of the variety but also on the group itself. Note that
$$
\mu^{\Gor}(\L(U_i)_0) = \LL^{-2} (1 + \LL) = \mu^{\orb}(\L(U_1)_0), \quad i=1,2,
$$
and that $\mu^{\orb}(\L(U_2)_0)$ corresponds in $U_1$ with the integral
$$
\int_{\L(U_1)_0^{\reg}} \LL^{-\ord_t x^{-1/2}} \dd \mu_{\L(U_1)}^{\Gor}.
$$
This is a consequence of Theorem~\ref{thm:intro_main1} and the fact that $\omega_{U_1}$ is induced by the $2$-form
$\dd x \wedge \dd y \in \Omega_{\CC^2}^2$ while $\omega_{U_2}$ is induced by $x \dd x \wedge \dd y$.

The difference between the two orbifold measures above is also related to the difference between two instances of Batyrev's \emph{stringy $E$-function}~\cite{Batyrev99}.
Let $X$ be either $U_1$ or $U_2$ and consider the projections $\CC^2\to U_i$, $i=1,2$. Using the notation of~\cite{Batyrev99}, the discriminant $\Delta_{U_1}=0$ while $\Delta_{U_2}=\frac{1}{2}\ell$ where $\ell$ is the divisor in $U_2$ given by $x=0$. Then $E_{\text{st}}(X,0;u,v)=H\left(\mu^{\orb}(\L(U_1)_0)\right)$ and $E_{\text{st}}(X,\frac{1}{2}\ell;u,v)=H\left(\mu^{\orb}(\L(U_2)_0)\right)$, where $E_{\text{st}}$ is the stringy $E$-function associated to a klt pair, and $H$ is the map induced by the Hodge realization $\KVarC\to\ZZ[u,v]$.
\end{ex}

\subsection{Quotient spaces and $\QQ$-resolutions of singularities}\mbox{}\label{sec:Qres}

Let us introduce some notation and notions in the context of $\QQ$-resolutions of singularities. We refer to~\cite{AMO14b} for the details, see also~\cite{Steenbrink77}.

\begin{defi}\label{defi:QNC}
Let $X$ be a $V$-manifold with abelian quotient singularities. A hypersurface $D$ on $X$ is said to have
\emph{$\QQ$-normal crossings} if it is locally analytically isomorphic to the quotient of a union of coordinate hyperplanes under
a diagonal action of a finite abelian group $G \subset \GL_n(\CC)$.
\end{defi}

Let $U=\CC^n/G$ be an abelian quotient space. Consider $H\subset U$ an analytic subvariety of codimension one.
\begin{defi}\label{defi:Qres}
	An \emph{embedded $\QQ$-resolution} of $(H,0)\subset (U,0)$ is a proper analytic map $\pi:Y\to (U,0)$ such that
	\begin{enumerate}
		\item $Y$ is a $V$-manifold with abelian quotient singularities,
		
		\item $\pi$ is an isomorphism over $Y\setminus \pi^{-1}(H_{\sing})$,
		
		\item $\pi^{-1}(H)$ is a hypersurface with $\QQ$-normal crossings on $Y$.
	\end{enumerate}
\end{defi}

\begin{remark}
	Let $(H,0)$ be the hypersurface defined by a non-constant analytic germ $f:(U,0)\to (\CC,0)$, and let
	$\pi:Y\to(U,0)$ be an embedded $\QQ$-resolution of $(H,0)$. Then $\pi^{-1}(H)$ is locally given by a function
	of the form
	\[
	x_1^{a_1}\cdots x_n^{a_n}:\CC^n/G_0\longrightarrow\CC,
	\]
	for some finite abelian group $G_0$. Moreover, there is a natural finite stratification
	$Y = \bigsqcup_{k \geq 0} Y_{k}$ such that the multiplicities $a_1,\ldots,a_n\in\NN$ and the action of $G_0$
	above is constant along each stratum $Y_k$, i.e.,~it does not depend on the chosen point $p\in Y_k$.
\end{remark}

For every finite group $G\subset\GL_n(\CC)$, denote by $G_{\text{big}}$ the normal subgroup generated by all rotations around hyperplanes. Then, the quotient space $\CC^n/G_{\text{big}}$ is isomorphic to $\CC^n$ since the $G_{\text{big}}$-invariant polynomials form a polynomial algebra. The choice of a basis in this algebra determines an isomorphism between the group $G/G_{\text{big}}$ and another group in $\GL_n(\CC)$ which does not contain non-trivial rotations around hyperplanes. The latter gives rise to the following notion.

\begin{defi}\label{def:small}
	A finite group $G\subset\GL_n(\CC)$ is called \emph{small} if no element of $G$ has 1 as an eigenvalue of multiplicity exactly $n-1$.
\end{defi}

It is well known that the conjugacy class of an small group $G\subset\GL_n(\CC)$ determines $\CC^n/G$ up to isomorphism, and vice versa.\\

For $\bd = {}^t(d_1 \ldots d_r)$ we denote by
$C_{\bd} = C_{d_1} \times \cdots \times C_{d_r}$ the finite
abelian group written as a product of finite cyclic groups, that is, $C_{d_i}$
is the cyclic group of $d_i$-th roots of unity in $\CC$. Consider a matrix of weight
vectors
\begin{align*}
	A & = (a_{ij})_{i,j} = [\ba_1 \, | \, \cdots \, | \, \ba_n ] \in \text{Mat}_{r \times n} (\ZZ), \\ %\quad
	\ba_j & = {}^t (a_{1 j}\dots a_{r j}) \in \text{Mat}_{r\times 1}(\ZZ),
\end{align*}
and the action
\begin{equation*}%\label{action_XdA}
	\begin{array}{cr}
	( C_{d_1} \times \cdots \times C_{d_r} ) \times \CC^n  \longrightarrow  \CC^n,&\xi_\bd = (\xi_{d_1}, \ldots, \xi_{d_r}), \\[0.15cm]
	\big( \xi_{\bd} , \mathbf{x} \big) \mapsto (\xi_{d_1}^{a_{11}} \cdot\ldots\cdot
	\xi_{d_r}^{a_{r1}}\, x_1,\, \ldots\, , \xi_{d_1}^{a_{1n}}\cdot\ldots\cdot
	\xi_{d_r}^{a_{rn}}\, x_n ), & \mathbf{x} = (x_1,\ldots,x_n).
	\end{array}
\end{equation*}

Note that the $i$-th row of the matrix $A$ can be considered modulo $d_i$. The
set of all orbits $\CC^n / C_\bd$ is called the ({\em cyclic}) {\em quotient space of
type $(\bd;A)$} and it is denoted by
$$
  X(\bd; A) = X \left( \begin{array}{c|ccc} d_1 & a_{11} & \cdots & a_{1n}\\
\vdots & \vdots & \ddots & \vdots \\ d_r & a_{r1} & \cdots & a_{rn} \end{array}
\right).
$$
When $n=2$ and $r=1$, the space $X(d;a,b)$ is customarily denoted by $\frac{1}{d}(a,b)$.
\begin{ex}
	With the notation of this section, $U_1$ and $U_2$ in Example \ref{ex:Gor-measure}, are written as $\frac{1}{2}(1,1)$ and $\frac{1}{4}(1,2)$, respectively.
\end{ex}

% ====================================================================
% Zeta function and change of variables formula
% ====================================================================
\bigskip
\section{Zeta function on \texorpdfstring{$\QQ$}{Q}-Gorenstein varieties
and change of variables formula}\label{sc:def_zeta_function}

Let $X$ be a $\QQ$-Gorenstein complex algebraic variety of pure dimension $n$, with at most log
terminal singularities. Consider two $\QQ$-Cartier divisors $D_1, D_2$ in $X$, and take an integer $r$ such that $r D_1$ and $rD_2$ are Cartier
and $\omega_X^{[r]}$ is invertible.
\begin{defi}\label{def:motivic_zeta_function2}
Let $W$ be a subvariety of $X$ and consider $\L(X)_W = \tau_0^{-1}(W)$. We also set
$\L(X)_W^{\reg}=\L(X)_W\setminus\L(X_{\sing})$. The \emph{motivic Gorenstein zeta function}
of the pair $(D_1,D_2)$ with respect to $W$ is
\[
\Zmot_{,W}(D_1,D_2; s) = \int_{\L(X)_{W}^{\reg}} \LL^{-(\ord_t D_1\cdot  s + \ord_tD_2)} \dd \mu^{\Gor}_{\L(X)},
\]
whenever the right-hand side converges in $\Mhatloc \llb\LL^{- s/r}\rrb$. Here $\LL^{- s}$ and $\LL^{- s/r}$ are just variables, where $\LL^{- s/r}$
should be understood as $(\LL^{- s})^{1/r}$.
The divisor $D_2$ can be thought of as a divisor associated
with a maximal degree form in $X$. When $W$ is just a point $P \in X$, the zeta function is simply called
the \emph{local motivic zeta function} at $P$ and it is denoted by $\Zmot_{,P}(D_1,D_2; s)$.
\end{defi}

Note that when $\Supp D_2\subset \Supp D_1$, then $\Zmot_{,W}(D_1,D_2; s)\in \Mloc \llb\LL^{- s/r}\rrb$,
see e.g.~\cite[Section~1.5]{Veys01}. In addition, if $D_1$ and $D_2 + K_X$ are Cartier and the integral
defining the zeta function converges, then $\Zmot_{,0}(D_1,D_2; s)$ is an element of
$\Mhatloc\llb\LL^{- s}\rrb$.

The definition of $\Zmot_{,W}(D_1,D_2; s)$ is a generalization of the classical motivic zeta function given
in~\cite{DL98,DL99}, or in~\cite{NemethiVeys12, CauwbergsVeys17}, where different $D_2$ associated with
$n$-differential forms over a smooth $X$ are considered.

The change of variables formula given in Theorem~\ref{thm:intro_orbs_chg_vars} is a special case
of~\cite[Theorem 1.16]{DL02} that is well suited for our Gorenstein zeta function. In particular, it can be considered as a version of formula \eqref{eq:chg_var_div} for $\QQ$-Gorenstein varieties.
Now we are ready to present its proof.

\begin{proof}[Proof of Theorem~\ref{thm:intro_orbs_chg_vars}]
Assume without loss of generality that $\omega_X^{[r]}$ and $\omega_Y^{[r]}$ are invertible,
that is, the same $r$ is valid for both sheaves. For convenience, we also denote by $\pi$ the map at the level of arc spaces $\L(Y)\to \L(X)$. We claim that (it will be proven later)
\begin{equation}\label{eq:jacobian-lemma}
- \frac{\ord_t \omega_X^{[r]}}{r} \circ \pi - \ord_t \Jac_{\pi} + \frac{\ord_t \omega_Y^{[r]}}{r}
= - \ord_t K_\pi.
\end{equation}

By definition of the Gorenstein measure, the zeta function $\Zmot_{,W}(D_1,D_2; s)$ of the pair
$(D_1,D_2)$ with respect to $W$ is
$$
\int_{\L(X)_{W}^{\reg}} \LL^{-(\ord_t D_1\cdot  s + \ord_tD_2)} \dd \mu^{\Gor}_{\L(X)}
= \int_{\L(X)_{W}^{\reg}} \LL^{-(\ord_t D_1\cdot  s + \ord_tD_2)} \LL^{-\frac{1}{r}\ord_t \omega_X^{[r]}} \dd \mu_{\L(X)}.
$$
Applying the change of variables Theorem~\ref{thm:chg_var_DL}, one gets the integral
$$
\int_{\pi^{-1} \L(X)_{W}^{\reg}} \LL^{-((\ord_t D_1\circ \pi) \cdot  s + (\ord_tD_2 \circ \pi))}
\LL^{-\frac{1}{r}\ord_t \omega_X^{[r]} \circ \pi - \ord_t \Jac_{\pi}} \dd \mu_{\L(Y)},
$$
which in terms of the Gorenstein measure on $\L(Y)$ becomes
\begin{equation}\label{eq:long-formula}
\int_{\pi^{-1} \L(X)_{W}^{\reg}} \LL^{-((\ord_t D_1\circ \pi) \cdot  s + (\ord_tD_2 \circ \pi))}
\LL^{-\frac{1}{r}\ord_t \omega_X^{[r]} \circ \pi - \ord_t \Jac_{\pi}+ \frac{1}{r}\ord_t \omega_Y^{[r]}}
\dd \mu^{\Gor}_{\L(Y)}.
\end{equation}
Since $\ord_t D_i \circ \pi = \ord_t \pi^{*} D_i$ ($i=1,2$),
$- \frac{\ord_t \omega_X^{[r]}}r \circ \pi - \ord_t \Jac_{\pi} + \frac{\ord_t \omega_Y^{[r]}}r
= - \ord_t K_\pi$, and $\pi^{-1} \L(X)_W^{\reg}$ is equal to $\L(Y)_{\pi^{-1}W}^{\reg}$ up to a set of measure
zero, the last expression~\eqref{eq:long-formula} equals $\Zmot_{,\pi^{-1}W}(\pi^{*}D_1,\pi^{*}D_2+K_\pi; s)$
and the result follows.

It remains to show~\eqref{eq:jacobian-lemma} -- although we believe this result is well known by
specialists, cf.~\cite[Lemma 2.16]{Yasuda04}, for the sake of completeness we also present the proof here.
Let $\psi$ be an arc on $Y$ and $\varphi = \pi(\psi)$.
Recall from Definition~\ref{def:order-jacobian} that we denote by $L_Y$ and $L_X$ the images of $\psi^{*} (\Omega_Y^n)$ and
$\psi^{*} \pi^{*} (\Omega_X^n)$ in $V=\psi^\ast(\Omega_Y^n)\otimes_{\CC\llb t\rrb}\CC((t))$, respectively, and that
$L_X = t^{\ord_t \Jac_\pi(\psi)}L_Y$. Hence $L_X^{\otimes r} = t^{r \ord_t \Jac_\pi(\psi)}L_Y^{\otimes r}$, as lattices in $V^{\otimes r}$.

Let further $\Lambda_Y$ and $\Lambda_X$ denote the images in $V^{\otimes r}$  of $\psi^{*} (\omega_Y^{[r]})$ and $\psi^{*}\pi^{*}(\omega_X^{[r]})$, respectively. Hence $\Lambda_Y=t^{\ord_t \omega_Y^{[r]}(\psi)}L_Y^{\otimes r}$ and $\Lambda_X=t^{\ord_t \omega_X^{[r]}(\varphi)}L_X^{\otimes r}$ in $V^{\otimes r}$.
Combining these equalities yields
\begin{equation}\label{eq:psi-phi-jacobian1}
\Lambda_Y= t^{(\ord_t \omega_Y^{[r]} - r \ord_t \Jac_\pi - \ord_t \omega_X^{[r]}\circ \pi )(\psi) } \Lambda_X .
\end{equation}

On the other hand, since $\omega_Y^{[r]} = \O_Y (rK_\pi) \otimes \pi^{*} \omega_X^{[r]}$, we have that
\begin{equation}\label{eq:psi-phi-jacobian2}
\Lambda_Y= t^{-r\ord_t K_\pi (\psi) } \Lambda_X
\end{equation}
(note that the minus sign is caused by the explicit correspondence between divisors and invertible sheaves).

Combining~\eqref{eq:psi-phi-jacobian1} and~\eqref{eq:psi-phi-jacobian2} finishes the proof.
\end{proof}

In the following example, we illustrate the relation between $\ord_t K_\pi$ and $\ord_t\Jac_{\pi}$ as in~(\ref{eq:jacobian-lemma}) above in the proof of Theorem~\ref{thm:intro_orbs_chg_vars}.

\begin{ex}\label{ex:jacobian}
Let $\pi: Y\to X=\CC^2$ be the weighted blow-up with weights $(2,3)$. Thus, $Y$ is a $V$-manifold with two singular points $P$ and $Q$ of type $\frac{1}{2}(1,1)$ and $\frac{1}{3}(1,1)$, respectively, located on the exceptional divisor $E=\pi^{-1}(0)$. In this case, the relative canonical divisor is $K_{\pi}=4E$.  Let us fix coordinates $(x,y)$ on $X$ around the origin and coordinates $(u,v)$ on $Y$ around $Q$, such that $\pi$ is given by the substitution $x=uv^2$ and $y=v^3$. More details about weighted blow-ups are given e.g.~in~\cite[Section~4.2]{AMO14b}. %

Consider the arc $\psi\in\L(Y)^{\reg}_Q$ induced by $(t^{k_1/3}, t^{k_2/3})$ where $k_1\equiv k_2\mod 3$. According to Definition~\ref{def:order-jacobian}, in order to compute $\ord_t\Jac_{\pi}(\psi)$, one needs to study the images in $V$ of $\psi^*\pi^*(\Omega_{X}^2)$ and $\psi^*(\Omega_{Y}^2)$, denoted by $L_X$ and $L_Y$, respectively.
Note that ``morally'' the Kähler differentials $\Omega_Y^2$ on $Y$ are locally generated around $Q$ by $\{u^iv^{4-i}\dd u\wedge\dd v\}_{i=0}^4$. (More precisely, the form $\dd u\wedge\dd v$ is not well defined on $Y$. For the exact computation one could work with $(\Omega_Y^2)^{\otimes 3}$, locally generated by $\{u^{3i}v^{3(4-i)}(\dd u\wedge\dd v)^{\otimes 3}\}_{i=0}^4$. But the formal computation with $\dd u\wedge\dd v$ is more practical and yields the same result.)  Similarly, since $X=\CC^2$, the image of $\pi^*\Omega_{X}^2$ in $\Omega_{Y}^2$  is locally generated by $v^4\dd u\wedge\dd v$.
Thus, fixing a suitable generator $g$ of the vector space $V$, we have that $L_Y$ is generated by $t^{m}g$, where $m=\min\set{\frac{k_1i+k_2(4-i)}{3} \mid i=0,\ldots,4}=\frac{4}{3}\min\{k_1,k_2\}$, and that $L_X$ is  generated by $t^{n}g$, where $n=\frac{4k_2}{3}$. Hence $L_X=t^e L_Y$, where  $$\ord_t\Jac_{\pi}(\psi)=e=n-m=\frac{4}{3}\left(k_2-\min\{k_1,k_2\}\right)\in\NN.$$
On the other hand, $\ord_t K_\pi(\psi)=\frac{4k_2}{3}$, which is much easier to calculate.

In order to verify~(\ref{eq:jacobian-lemma}), note that $\ord_t\omega_X^{[3]}$ is the zero function and that $\ord_t\omega_Y^{[3]}(\psi) = -4\min\{k_1,k_2\}$, using the calculation above and the fact that $(\dd u\wedge\dd v)^{\otimes 3}$ generates $\omega_Y^{[3]}$.
\end{ex}

% ====================================================================
% Quotient singularities under finite abelian groups
% ====================================================================
\bigskip
\section{Quotient singularities under finite abelian groups}
\label{sc:abelian_quotient_sing}

The main aim of this section is to prove Theorem~\ref{thm:intro_main1} and Theorem~\ref{thm:intro_main2}.
Recall the notation presented in the introduction and in Section~\ref{sc:arc-spaces}. Let $U = \CC^n/G$ where $G \subset \GL_n(\CC)$ is a finite abelian
small subgroup acting diagonally as in~\eqref{eq:our_diag_matrix}. Assume $D_1$ and $D_2$ are $\QQ$-divisors
given by $x_1^{N_1} \cdots x_n^{N_n}$ and $x_1^{\nu_1-1} \cdots x_n^{\nu_n-1}$ and denote
$\Nbf = (N_1,\ldots,N_n), \nubf = (\nu_1,\ldots,\nu_n) \in \frac{1}{d} \ZZ_{\geq 0}^n$, where $d = |G|$.
This means that $dD_1$ and $dD_2$ are Cartier. Given $\kbf \in \QQ^n$ and $\gamma \in G$ we refer
to Section~\ref{sc:arc-spaces} for the definition of $w_{\kbf}(\gamma)$
and $\varpi_{\kbf}(\gamma)$ and the relation between them. In particular, the notation
$\one = (1,\ldots,1)$ and $\zero = (0,\ldots,0)$ will be used below.

\begin{proof}[Proof of Theorem~\ref{thm:intro_main1}.]
Let us first use the decomposition $\L(U)_0^{\reg} = \bigsqcup_{\gamma\in G} \L(U)_{0,\gamma}^{\reg}$
given in~\eqref{eq:decomposition-arcs} so that $\Zmot_{,0}(D_1,D_2;s)$ splits into $d$ different integrals
and thus, for a fixed $\gamma \in G$, it is enough to compute
\begin{equation}\label{eq:step0}
\int_{\L(U)_{0,\gamma}^{\reg}} \LL^{-(\ord_t D_1 \cdot  s + \ord_t D_2 )} \dd\mu^{\Gor}_{\L(U)}.
\end{equation}

Consider the morphism $\lambda_\gamma: U \otimes \CC[t] \to U \otimes \CC[t]$ described in~\eqref{eq:lambda}
that induces a $\CC[t]$-morphism on the arc spaces $\lambda_\gamma: \L(U) \to \L(U)$ defined by
$\left[(\varphi_1(t),\ldots,\varphi_n(t))\right] \mapsto \left[ (t^{e_{\gamma,1}/d}\varphi_1(t),
\ldots,t^{e_{\gamma,n}/d}\varphi_n(t)) \right]$. Note that $\lambda^{*}_{\gamma} D_1 = t^{w_{\Nbf}(\gamma)} D_1$,
$\lambda^{*}_{\gamma} D_2 = t^{w_{\nubf-\one}(\gamma)} D_2$, and $\lambda^{*}_{\gamma} \omega_U = t^{w_\one(\gamma)} \omega_U$,
since $G$ a small subgroup of $\GL_n(\CC)$. Using now that
\begin{enumerate}
	\item $\ord_t D_i \circ \lambda_{\gamma} = \ord_t \lambda_{\gamma}^{*} D_i$ ($i=1,2$),
	
	\item analogous arguments as in the proof of~(\ref{eq:jacobian-lemma}) yield
	$$
	- \frac{\ord_t \omega_U^{ [d] }}{d} \circ \lambda_{\gamma} - \ord_t \Jac_{\lambda_{\gamma}}
	+ \frac{\ord_t \omega_U^{ [d] }}{d} = - \frac{e_{\gamma,1} + \cdots + e_{\gamma,n}}{d} = - w_{\one}(\gamma),
	$$
	
	\item $w_{\nubf-\one}(\gamma) + w_{\one}(\gamma) = w_{\nubf}(\gamma)$,
\end{enumerate}
\noindent applying Theorem~\ref{thm:chg_var_DL} to $\lambda_{\gamma}$ converts the integral~\eqref{eq:step0} into
\begin{equation}\label{eq:common-factor}
\LL^{-(w_\Nbf(\gamma)\cdot s + w_\nubf(\gamma))} \int_{\L(\CC^n)/G} \LL^{-(\ord_t D_1 \cdot  s + \ord_t D_2)}
\dd\mu^{\Gor}_{\L(U)}.
\end{equation}

Now, as in the case of $\CC^n$, the idea of the proof is to decompose $\L(\CC^n)/G$ into different pieces so that
$\ord_t D_1$ and $\ord_t D_2$ remain constant and hence the computation of the previous integral is reduced to
the computation of the $\QQ$-Gorenstein measure of the corresponding pieces.

To do so, let us consider the natural partition on $\L(\CC^n)$ given by the order function on each coordinate.
This partition is compatible with the group action and it gives rise to a partition on $\L(\CC^n)/G$. That is,
\begin{equation}\label{eqn:partition_LCG}
\L(\CC^n)/G = A_{\infty} \sqcup \bigsqcup_{\kbf \in \ZZ_{\geq0}^n} A_{\kbf},
\end{equation}
where $\kbf = (k_1,\ldots,k_n)$, $A_{\kbf} = \{ [\varphi] = [(\varphi_1,\ldots,\varphi_n)] \in \L(\CC^n)/G
\mid \ord_t (\varphi_i) = k_i,\, \forall i \}$ and $A_{\infty}$ is the complement of
$\bigsqcup_{\kbf\in\ZZ_{\geq0}^n} A_{\kbf}$ in $\L(\CC^n)/G$. Since $A_{\infty}$ consists of all the elements
$[\varphi]$ such that $\varphi_i = 0$ for some $i = 1,\ldots,n$, its measure is zero.

\vspace{0.35cm}

{\sc Assertion. $\mu^{\Gor}_{\L(U)}(A_{\kbf})=\LL^{-n-\sum_i k_i}(\LL-1)^n$}.

\begin{proof}[Proof of the assertion.]
For a given $\kbf \in \ZZ_{\geq 0}^n$, consider the morphism $\lambda_{\kbf}: U \otimes \CC[t] \to U \otimes \CC[t]$
that induces a $\CC[t]$-morphism on the arc spaces $\lambda_{\kbf}: \L(U) \rightarrow \L(U)$ defined by
$\left[ (\varphi_1(t), \ldots, \varphi_n(t)) \right] \mapsto \left[ (t^{k_{1}}\varphi_1(t),\ldots,t^{k_{n}}\varphi_n(t)) \right]$.
This map satisfies
$$
- \frac{\ord_t \omega_U^{ [d] }}{d} \circ \lambda_{\kbf} - \ord_t \Jac_{\lambda_{\kbf}}
+ \frac{\ord_t \omega_U^{ [d] }}{d} = - \sum_{i=1}^n k_i,
$$
again by analogous arguments as before, and it provides a bijection between $A_{\zero}$ and $A_{\kbf}$. Using Theorem~\ref{thm:chg_var_DL} one obtains the relation
$$
\mu^{\Gor}_{\L(U)}(A_{\kbf})
 =\int_{A_{\kbf}}  \dd\mu^{\Gor}_{\L(U)}
 = \LL^{-\sum_i k_i} \int_{A_{\zero}} \dd \mu^{\Gor}_{\L(U)}
 = \LL^{-\sum_i k_i} \mu^{\Gor}_{\L(U)}(A_{\zero}).
$$
To compute the measure of $A_{\zero}$, we use partition~\eqref{eqn:partition_LCG}
and the previous relation:
$$
\mu^{\Gor}_{\L(U)}(\L(\CC^n)/G)  = \sum_{\kbf\in\ZZ_{\geq0}^n} \mu^{\Gor}_{\L(U)}(A_{\kbf})
= \mu^{\Gor}_{\L(U)}(A_{\zero}) \sum_{\kbf\in\ZZ_{\geq0}^n} \LL^{-\sum_i k_i}
=  \mu^{\Gor}_{\L(U)}(A_{\zero})\prod_{i=1}^{n}\frac{1}{1-\LL^{-1}}.
$$
The fact that $\L(\CC^n)/G$ has $\QQ$-Gorenstein measure $1$, see~\cite[Lemma 3.4]{DL02}, concludes the proof of
the assertion. Note that the latter is true without adding any relation in the Grothendieck ring
since the group $G$ is abelian, see~\cite[Lemma 5.1]{Looijenga02}.
\end{proof}

We continue with the proof of Theorem~\ref{thm:intro_main1}. On $A_{\kbf}$ the order of $D_1$ is $\sum_i k_i N_i$ and the order of $D_2$ is $\sum_i k_i (\nu_i-1)$.
From the partition of $\L(\CC^n)/G$ in~\eqref{eqn:partition_LCG} and the assertion above, one gets
\begin{equation}\label{eq:int-LCnG}
\begin{aligned}
& \int_{\L(\CC^n)/G} \LL^{-(\ord_t D_1 \cdot  s + \ord_t D_2)} \dd \mu^{\Gor}_{\L(U)}
= \sum_{\kbf\in\ZZ_{\geq0}^n} \LL^{-\sum_{i=1}^{n} k_i \left(N_i\cdot s + \nu_i-1 \right)} \mu^{\Gor}_{\L(U)}(A_{\kbf}) \\
&= \LL^{-n}(\LL-1)^n \sum_{\kbf\in\ZZ_{\geq0}^n} \LL^{-\sum_{i=1}^{n}k_i\left(N_i\cdot s + \nu_i\right)}
= \LL^{-n}(\LL-1)^n\prod_{i=1}^{n}\frac{1}{1-\LL^{-(N_i\cdot s+\nu_i)}}.
\end{aligned}
\end{equation}

Due to the relations $-w_{\Nbf}(\gamma) = \varpi_{\Nbf}(\gamma^{-1}) - \sum_{i} N_i$ and
$-w_{\nubf}(\gamma) = \varpi_{\nubf}(\gamma^{-1}) - \sum_{i} \nu_i$, coming from~\eqref{eq:weights},
to conclude one simply observes that
\begin{equation}\label{eq:SGNnu}
\sum_{\gamma\in G} \LL^{- \left( w_{\Nbf}(\gamma) \cdot  s + w_{\nubf}(\gamma) \right) }
= \sum_{\gamma\in G} \LL^{\varpi_{\Nbf}(\gamma^{-1}) \cdot  s + \varpi_{\nubf}(\gamma^{-1})
- \sum_i (N_i \cdot  s + \nu_i)}
= S_G(\Nbf,\nubf;s) \LL^{- \sum_i (N_i \cdot  s + \nu_i)}.
\end{equation}
According to~\eqref{eq:common-factor}, one multiplies~\eqref{eq:int-LCnG} and~\eqref{eq:SGNnu} to obtain the desired
formula for the zeta function $\Zmot_{,0}(D_1,D_2; s)$.
\end{proof}

\begin{remark}\label{rk:comparison-DL}
If the group $G \subset \GL_n(\CC)$ is not small, then Theorem~\ref{thm:intro_main1} remains true replacing
$\Zmot_{,0}(D_1,D_2;s)$ by the integral
\begin{equation}\label{eq:non-small-groups}
\int_{\L(U)_{0}^{\reg}} \LL^{-(\ord_t D_1\cdot  s + \ord_t D_2)} \LL^{- \ord_t (\dd x_1 \wedge \cdots \wedge \dd x_n) }
\dd \mu_{\L(U)}.
\end{equation}
Therefore the case $D_1$ and $D_2$ equal to zero,
or equivalently $\Nbf = \zero$ and $\nubf = \one$, gives the relation
\begin{equation}\label{eq:orb-measure}
\int_{\L(U)_{0}^{\reg}} \LL^{- \ord_t (\dd x_1 \wedge \cdots \wedge \dd x_n) } \dd \mu_{\L(U)}
= S_G(\zero,\one;s) \LL^{-n}  = \sum_{\gamma \in G} \LL^{\frac{1}{d}(\varepsilon_{\gamma,1}
+ \cdots + \varepsilon_{\gamma,n})-n}.
\end{equation}
In the notation of~\cite[Section 3.7]{DL02}, $\ord_t (\dd x_1 \wedge \cdots \wedge \dd x_n) = \alpha_U$
and thus~\cite[Theorem 3.6]{DL02} is nothing but the previous formula~\eqref{eq:orb-measure}
which calculates the orbifold measure of $\L(U)_0^{\reg}$.
\end{remark}

\begin{proof}[Proof of Theorem~\ref{thm:intro_main2}]
Consider the notation right before the statement of Theorem~\ref{thm:intro_main2}.
First let us apply Theorem~\ref{thm:intro_orbs_chg_vars} to $\pi: Y \to X$.
Then $\Zmot_{,W} (D_1,D_2;s) = \Zmot_{,\pi^{-1} W} ( \pi^{*} D_1, \pi^{*} D_2 + K_{\pi}; s)$.
The stratification on $Y = \bigsqcup_{k \geq 0} Y_k$ provides a stratification
$\pi^{-1} W = \bigsqcup_{k \geq 0} (Y_k \cap \pi^{-1} W)$ that allows one to decompose
the zeta function into several summands:
$$
\Zmot_{,\pi^{-1} W} ( \pi^{*} D_1, \pi^{*} D_2 + K_{\pi}; s)
= \sum_{k \geq 0} \Zmot_{,Y_k \cap \pi^{-1} W} ( \pi^{*} D_1, \pi^{*} D_2 + K_{\pi}; s).
$$

Let $q_k \in Y_k \cap \pi^{-1} W$ be an arbitrary point. Since the equations of $\pi^{*} D_1$
and $\pi^{*} D_2 + K_{\pi}$ do not depend on the chosen point $q_k \in Y_k$ but only on the
stratum $Y_k$ \,-- in fact they are given by $x_1^{N_{1,k}} \cdots x_n^{N_{n,k}}$ and
$x_1^{\nu_{1,k}-1} \cdots x_n^{\nu_{n,k}-1}$ on $U_k = \CC^n / G_k$ --\,
standard arguments~\cite[Thm.~2.15]{Craw04} show that
$$
\Zmot_{,Y_k \cap \pi^{-1} W} ( \pi^{*} D_1, \pi^{*} D_2 + K_{\pi}; s)
= [Y_k \cap \pi^{-1} W] \Zmot_{,q_k} ( \pi^{*} D_1, \pi^{*} D_2 + K_{\pi}; s).
$$
Now the result follows from the monomial case, i.e., Theorem~\ref{thm:intro_main1}, applied to each stratum $Y_k$, $k \geq 0$.
\end{proof}

\begin{remark}\label{rk:comparison-Veys}
This result was already stated in~\cite[Theorem 6.1]{Veys97} for the Igusa zeta function
in the case of $X=\CC^2$, $D_1$ an effective Cartier divisor, and $D_2=0$.
In loc.~cit., the sum $S_{G_k}(\Nbf_k,\nubf_k;s)$ appears as the (quite complicated) determinant $D_r$
of a $q$-deformation of the intersection matrix of the resolution. However, our $S_{G_k}(\Nbf_k,\nubf_k;s)$ is
a simple explicit formula in terms of the local action of $G_k$ on each stratum.
For instance, if $Y_k=\frac{1}{d}(a,b)$ with $\gcd(d,a,b)=1$ and the divisors $D_1$ and $D_2$ are given
by $x^{N_1} y^{N_2}$ and $x^{\nu_1-1} y^{\nu_2-1}$, then the corresponding sum can simply be expressed as
\begin{equation}\label{eq:sum-dim2}
S_{G}(\Nbf,\nubf;s) = \sum_{i=0}^{d-1} \LL^{\frac{1}{d}\left( (\overline{ia} N_1 + \overline{ib} N_2 )  s
+ (\overline{ia} \nu_1 + \overline{ib} \nu_2) \right)},
\end{equation}
where $\overline{c}$ denotes the class of $c$ modulo $d$ satisfying $0 \leq \overline{c} \leq d-1$.
In Example~\ref{ex:comparison-Veys} below the term $D_r$ is calculated for a concrete case.
\end{remark}

\begin{ex}\label{ex:comparison-Veys}
Let $X = \frac{1}{7}(1,3)$ and consider $f = x^{N_1} y^{N_2}$ and $g = x^{\nu_1-1} y^{\nu_2-1}$, as well as the form
$\omega = g \dd x \wedge \dd y$ such that $f \in \mathcal{O}_X$ and $\omega \in \Omega_X^2$.
Let $D_1$ and $D_2$ be the divisors associated with $f$ and $g$.
According to~\eqref{eq:sum-dim2}, we have
\begin{equation}\label{eq:sum-713}
\begin{aligned}
S_G(\Nbf,\nubf;s) = 1 &+ \LL^{\frac{N_1+3N_2}{7}s+\frac{\nu_1+3\nu_2}{7}} + \LL^{\frac{2N_1+6N_2}{7}s+\frac{2\nu_1+6\nu_2}{7}}
+ \LL^{\frac{3N_1+2N_2}{7}s+\frac{3\nu_1+2\nu_2}{7}} \\
& + \LL^{\frac{4N_1+5N_2}{7}s+\frac{4\nu_1+5\nu_2}{7}} + \LL^{\frac{5N_1+N_2}{7}s+\frac{5\nu_1+\nu_2}{7}}
+ \LL^{\frac{6N_1+4N_2}{7}s+\frac{6\nu_1+4\nu_2}{7}}.
\end{aligned}
\end{equation}
Denote $A=\{x=0\}$ and $B=\{y=0\}$.
Now one resolves the singularity $\pi:\widehat{X} \to X$ using the Hirzebruch-Jung method
and obtains $r=3$ divisors $E_1$, $E_2$, $E_3$. In Figure~\ref{fig:resolution-713},
the numerical data associated with each $E_i$ are given in the format $(N,\nu,\kappa)$,
where $N$ is the multiplicity of $\pi^{*}D_1$, $\nu-1$ is the multiplicity of $\pi^{*}D_2$,
and~$\kappa$ is minus the self-intersection number of the corresponding divisor.
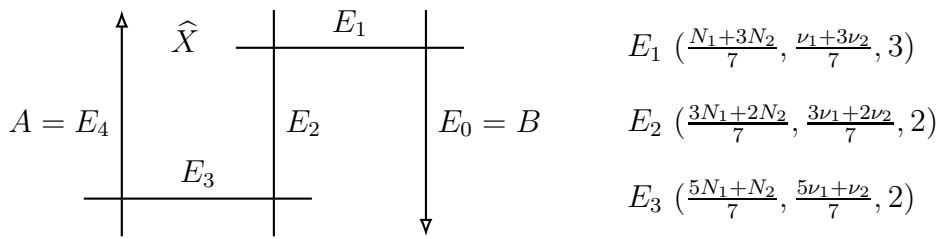
\begin{figure}[ht]
\begin{tikzpicture}[scale=1]
%\draw[help lines,step=1] (0,0) grid (5,3);
\draw[-{Latex[open]},thick] (0.5,0) -- node[left]{$A=E_4$} (0.5,3);
\draw[thick] (0,0.5) -- node[above]{$E_3$} (3,0.5);
\draw[thick] (2.5,0) -- node[right]{$E_2$} (2.5,3);
\draw[thick] (2,2.5) -- node[above]{$E_1$} (5,2.5);
\draw[-{Latex[open]},thick] (4.5,3) -- node[right]{$E_0=B$} (4.5,0);
\node[below right] at (1,3) {$\widehat{X}$};
\node[anchor=west] at (7,2.5) {$E_1 \ (\frac{N_1+3N_2}{7}, \frac{\nu_1+3\nu_2}{7}, 3)$};
\node[anchor=west] at (7,1.5) {$E_2 \ (\frac{3N_1+2N_2}{7}, \frac{3\nu_1+2\nu_2}{7}, 2)$};
\node[anchor=west] at (7,0.5) {$E_3 \ (\frac{5N_1+N_2}{7}, \frac{5\nu_1+\nu_2}{7}, 2)$};
\end{tikzpicture}
\caption{Resolution of $X=\frac{1}{7}(1,3)$ and numerical data of $f$ and $\omega$.}
\label{fig:resolution-713}
\end{figure}

By definition the matrix $\mathcal{D}_r$ in~\cite[Definition 5.5]{Veys97} is
$$
\mathcal{D}_r = \begin{vmatrix}
K_1 & -\LL^{<3>} & \LL^{<2>}-1 \\
-\LL^{<0>} & K_2 & -\LL^{<1>} \\
0 & -\LL^{<4>} & K_3
\end{vmatrix},
$$
where $K_1 = 1 + \LL^{<1>} + \LL^{2<1>}$, $K_2 = 1 + \LL^{<2>}$, $K_3 = 1 + \LL^{<3>}$,
and $<\!i\!>$ are given by
$$
\begin{aligned}
<\!1\!> &= \frac{N_1+3N_2}{7}s + \frac{\nu_1+3\nu_2}{7}, \\
<\!2\!> &= \frac{3N_1+2N_2}{7}s + \frac{3\nu_1+2\nu_2}{7}, \\
<\!3\!> &= \frac{5N_1+N_2}{7}s + \frac{5\nu_1+\nu_2}{7}, \\
\end{aligned} \qquad
\begin{aligned}
<\!0\!> &= N_2 s + \nu_2, \\
<\!4\!> &= N_1 s + \nu_1.
\end{aligned}
$$
A straightforward calculation shows that the determinant $\mathcal{D}_r$ and the expression
given in~\eqref{eq:sum-713} coincide.
\end{ex}

As in the classical case, one defines the \emph{topological zeta function} $Z_{\topo,W}(D_1,D_2;s)$ applying the Euler
characteristic specialization map to the motivic zeta function. Since $\chi (S_G(\Nbf,\nubf;s)) = |G|$, one easily obtains
a version of Theorem~\ref{thm:intro_main1} and Theorem~\ref{thm:intro_main2} for the topological zeta function
as follows.

\begin{cor}\label{cor:top1}
Under the assumptions of Theorem~\ref{thm:intro_main1} one has
\[
Z_{\topo,0}(D_1,D_2; s) = d \prod_{i=1}^{n} \frac{1}{N_i s+\nu_i} \in \QQ( s).
\]
\end{cor}

\begin{cor}\label{cor:top2}
Using the notation of Theorem~\ref{thm:intro_main2}, one has
\[
Z_{\topo,W}(D_1,D_2; s) = \sum_{k \geq 0} \chi \left(
Y_k \cap \pi^{-1}(W) \right) d_k \prod_{i = 1}^{n} \frac{1}{N_{i,k}  s + \nu_{i,k}}
\in \QQ( s),
\]
where $d_k$ denotes the order of the group $G_k$.
\end{cor}

% ====================================================================
% Yomdin surface singularities
% ====================================================================
\bigskip
\section{A family  of surface singularities}\label{sc:Yomdin_Qres}

Let $f:\CC^3\to \CC$ be the function defined by $f = z^{m+k} + h_m(x,y,z)$, where $h_m(x,y,z)$
is a homogeneous polynomial of degree $m$ in three variables and $k \geq 1$.
Assume that $C_0 = \{h_m = 0\} \subset \PP^2$ has only one singular point $P = (0:0:1)$,
analytically isomorphic to the cusp $x^q+y^p$ with $\gcd(p,q)=1$. Denote $k_1 = \gcd(k,p)$
and $k_2 = \gcd(k,q)$. Since $\Sing(C_0)\cap \{z=0\} = \emptyset$ in~$\PP^2$, the locus $D_1 = \{f=0\} \subset (\CC^3,0)$
defines a Yomdin surface singularity~\cite{Yomdin74}. In \cite[Example~3.7]{Martin11} the characteristic
polynomial of the monodromy associated with $D_1$ was computed via an embedded $\QQ$-resolution
$\pi: Y \to (\CC^3,0)$ of $D_1$, cf.~\cite{GLM97, Martin12}.

Let us briefly describe the resolution $\pi$ of \cite{Martin11}. Consider the classical blow-up at the origin
$\pi_1:\widehat{\CC}^3 \to \CC^3$. The third chart of $\pi_1$ is given by $(x,y,z) \mapsto (xz,yz,z)$ and the local equation
of the total transform is
$$
  z^m (z^k + x^q + y^p) = 0.
$$
The strict transform $\widehat{D}_1$ and the exceptional divisor $E_0$ intersect transversely
at every point but in $P\in C_0 \cong E_0\cap \widehat{D}_1$, see Figure~\ref{fig:E0}.
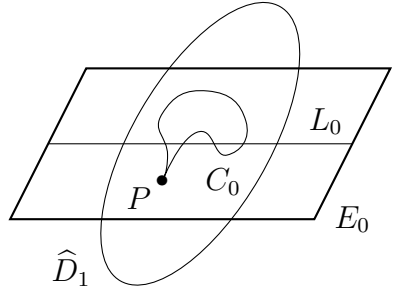
\begin{figure}[ht]
\centering
\scalebox{1}{
	\begin{tikzpicture}
	       % Rectangle
		\draw[thick] (-2.5,-1) -- (-1.5,1) -- (2.5,1) -- (1.5,-1) -- cycle;
		\node at (2,-1) {$E_0$};
		
		%L0
		\draw[] (-2,0) -- (2,0);
		\node at (1.65,0.3) {$L_0$};
		
		%Ellipse C0
		\draw[rotate=60] (0,0) ellipse (2.1cm and 0.9cm);
		\node at (-1.7,-1.65) {$\widehat{D}_1$};
		
		% Path
		\node (P0) at (-0.5,-0.5) {\footnotesize $\bullet$}; \node (P1) at (-0.05,0.15) {};
		\node (P2) at (0.35,-0.15) {}; \node (P3) at (0.6,0.3) {};
		\node (P4) at (0.1,0.7) {}; \node (P5) at (-0.5,0.4) {};
		\node (P6) at (-0.43,-0.15) {};
		%\draw [gray!50]  (P0.center) -- (P1.center)
		% -- (P2.center) -- (P3.center) -- (P4.center) -- (P5.center) -- (P6.center) -- (P0.center); %Control polygon
		\draw  plot [smooth, tension=1] coordinates {(P0) (P1) (P2) (P3) (P4) (P5) (P6) (P0)};
		\node[left] at (P0.south) {$P$};
		
		%C0
		\node at (0.3,-0.5) {$C_0$};
	\end{tikzpicture}
}
\caption{The divisor $E_0$ after the first blow-up.}
\label{fig:E0}
\end{figure}
Let $\pi_2$ be the weighted blow-up at $P$ with respect to the vector
$w = \left(\frac{k p}{k_1 k_2}, \frac{k q}{k_1 k_2}, \frac{p q}{k_1 k_2}\right)$.
The second chart of $\pi_2$ is given by $$(x,y,z) \longmapsto
\left(xy^{\frac{kp}{k_1k_2}},y^{\frac{kq}{k_1k_2}},y^{\frac{pq}{k_1k_2}}z\right)$$
and the local equation of the total transform is
\begin{equation}\label{eq:2nd-chart}
  \left\{ y^{\frac{p q}{k_1 k_2}(m+k)} z^m (z^k + x^q + 1) = 0\right\} \subset
  X\left(\frac{k q}{k_1 k_2}; \frac{k p}{k_1 k_2}, -1, \frac{p q}{k_1
k_2}\right),
\end{equation}
where $y=0$ represents the new exceptional divisor $E_1$.
The composition $\pi = \pi_1 \circ \pi_2$ is an embedded $\QQ$-resolution.
The total transform is given by
$$
\pi^* D_1 = \widehat{D}_1 + m E_0 + \frac{pq}{k_1k_2}(m+k) E_1
$$
and every cyclic quotient singularity is represented by a small group.
The final situation in $E_1$ is illustrated in Figure~\ref{fig:E1}.

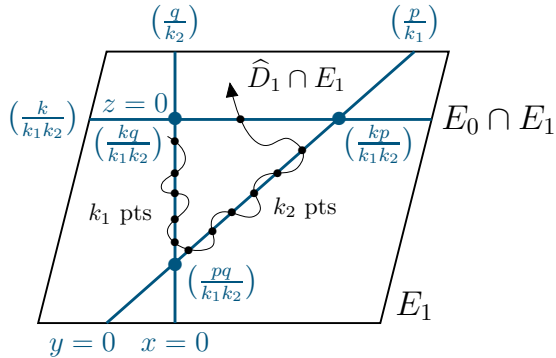
\begin{figure}[ht]
\centering
\scalebox{0.9}{
	\begin{tikzpicture}
		\usetikzlibrary{arrows}
		\definecolor{bluegreen2}{RGB}{0, 85, 127}
	       \colorlet{colsing}{bluegreen2}
	       \colorlet{colsingtext}{colsing}
		
	       % Rectangle
		\draw[thick] (-3,-1.5) -- (-2,2.5) -- (3,2.5) -- (2,-1.5) -- cycle;
		\node at (2.5,-1.25) {\large $E_1$};
		
		%z=0
		\node (Lz0) at (-2.25,1.5) {};
		\node (Lz1) at (2.75,1.5) {};
		\draw[very thick, color=colsing] (Lz0.center) -- (Lz1.center);
		\node[right] at (Lz1) {\large $E_0\cap E_1$};
		\node[left, text=colsingtext] at (Lz0) { $\big(\frac{k}{k_1k_2}\big)$};
		\node[right, xshift=1, yshift=3, text=colsingtext] at (Lz0.north) { $z=0$};
		
		%x=0
		\node (Lx0) at (-1,2.5) {};
		\node (Lx1) at (-1,-1.5) {};
		\draw[very thick, color=colsing] (Lx0.center) -- (Lx1.center);
		\node[below, text=colsingtext] at (Lx1) {$x=0$};
		\node[above, text=colsingtext] at (Lx0) { $\big(\frac{q}{k_2}\big)$};
		
		%y=0
		\node (Ly0) at (2.5,2.5) {};
		\node (Ly1) at (-2,-1.5) {};
		\draw[very thick, color=colsing] (Ly0.center) -- (Ly1.center);
		\node[below, xshift=-10, text=colsingtext] at (Ly1) {$y=0$};
		\node[above, text=colsingtext] at (Ly0) { $\big(\frac{p}{k_1}\big)$};
		
		%Points axis
		\node[text=colsing] (Pxz) at (-1,-0.65) {\large $\bullet$};
		\node[text=colsing,right] at (Pxz.south) { $\big(\frac{pq}{k_1k_2}\big)$};
		
		\node[text=colsing] (Pyz) at (-1,1.5) {\large $\bullet$};
		\node[text=colsing, left, yshift=-3] at (Pyz.south) { $\big(\frac{kq}{k_1k_2}\big)$};
		
		\node[text=colsing] (Pxy) at (1.4,1.5) {\large $\bullet$};
		\node[text=colsing,right, xshift=-3, yshift=-3] at (Pxy.south) { $\big(\frac{kp}{k_1k_2}\big)$};
		
		%Others
		\node at (-1.8,0.1) {\footnotesize $k_1$ pts};
		\node at (0.9,0.2) {\footnotesize $k_2$ pts};
		
		% Path
		\node (P0) at (-1.11,1.25) {}; \node (P1) at (-0.85,0.9) {};
		\node (P2) at (-1.2,0.5) {}; \node (P3) at (-0.8,0.3) {};
		\node (P4) at (-1.1,-0.2) {}; \node (P5) at (-0.5,-0.4) {};
		\node (P6) at (-0.4,0.1) {}; \node (P7) at (0.2,0.1) {};
		\node (P8) at (0.2,0.6) {}; \node (P9) at (0.8,0.8) {};
		\node (P10) at (0.7,1.2) {}; \node (P11) at (0.15,1.2) {};
		\node (Pfinal) at (-0.2,2.05) {};
		%\draw [gray!50]  (P0.center) -- (P1.center) -- (P2.center) -- (P3.center) -- (P4.center) -- (P5.center) -- (P6.center) -- (P7.center) -- (P8.center) -- (P9.center) -- (P10.center) -- (P11.center) -- (Pfinal.center); %Control polygon
		\draw[-triangle 45]  plot [smooth, tension=1] coordinates {(P0) (P1) (P2) (P3) (P4) (P5) (P6) (P7) (P8) (P9) (P10) (P11) (Pfinal)};
		\node[right, xshift=3, yshift=3] at (Pfinal) {$\widehat{D}_1\cap E_1$};
		
		%Points intersection
		\node at (-1,1.17) {\tiny $\bullet$};
		\node at (-1,0.7) {\tiny $\bullet$};
		\node at (-1,0.4) {\tiny $\bullet$};
		\node at (-1,0.03) {\tiny $\bullet$};
		\node at (-1,-0.32) {\tiny $\bullet$};
		
		\node at (-0.8,-0.43) {\tiny $\bullet$};
		\node at (-0.45,-0.15) {\tiny $\bullet$};
		\node at (-0.17,0.12) {\tiny $\bullet$};
		\node at (0.16,0.4) {\tiny $\bullet$};
		\node at (0.5,0.7) {\tiny $\bullet$};
		\node at (0.86,1.04) {\tiny $\bullet$};
		\node at (-0.04,1.5) {\tiny $\bullet$};
		
	\end{tikzpicture}
}
\caption{Intersection of $E_1$ with the other components.}
\label{fig:E1}
\end{figure}

The generalized A'Campo's formula~\cite{Martin11} provides the characteristic polynomial
of the monodromy associated with $D_1 \subset (\CC^3,0)$:
\begin{equation}\label{eq:monodromy}
\Delta(t) = \frac{\big(t^m-1\big)^{\chi(\PP^2\setminus C_0)}}{t-1} \cdot
\frac{\big(t^{m+k}-1\big)\big(t^{\frac{p q}{k_1 k_2}(m+k)}-1\big)^{k_1 k_2}}
{\big(t^{\frac{p}{k_1}(m+k)}-1\big)^{k_1}
\big(t^{\frac{q}{k_2}(m+k)}-1\big)^{k_2}},
\end{equation}
where $\chi(\PP^2 \setminus C_0) = m^2-3m+3-(p-1)(q-1)$. The Milnor number of $D_1$ is the degree
of $\Delta(t)$, that is, $\mu_{D_1} = (m-1)^3 + k (p-1)(q-1)$.

Let us compute the local motivic zeta function at the origin $\Zmot_{,0}(D_1,D_2; s)$
associated with $D_1$ above and $D_2 = (a-1)L$ with $a \geq 1$, where $L$ is the divisor defined
by a generic linear homogeneous polynomial in $\CC^3$. One checks that
$$
\pi^* D_2 + K_\pi = (a-1) \widehat{L} + (a+2-1) E_0 +  (\nu_1-1) E_1,
$$
where $\nu_1 = \frac{kp}{k_1k_2} + \frac{kq}{k_1k_2} + \frac{pq}{k_1k_2}(a+2)$.
Note that the morphism $\pi: Y \to (\CC^3,0)$ also defines a $\QQ$-resolution of $D_1+D_2$.

Let us describe the stratification $\bigsqcup_{\ell \geq 0} Y_{\ell}$ of $Y \cap \pi^{-1}(0) = E_0 \cup E_1$, as given by Theorem~\ref{thm:intro_main2}. The divisor $E_0$, created as a projective plane, does not contain any singular
point of $Y$. The curve $C_0 = \widehat{D}_1 \cap E_0$ of degree $m$ and the line $L_0 = \widehat{L} \cap E_0$ have to be
considered for the stratification of $E_0 \setminus \{ P \} = Y_0 \sqcup Y_1 \sqcup Y_2 \sqcup Y_3$.
In the following table we calculate the numerical data associated with each stratum.
$$
\begin{array}{l|c|c|c|c}
\multicolumn{1}{c|}{\text{Stratum}} & \text{Class} & \Nbf_{\ell} & \nubf_{\ell} & G_{\ell} \\
\hline
Y_0 = E_0 \setminus ( C_0 \cup L_0 ) & \LL^2 - [C_0] + m & (0,0,m) & (1,1,a+2) & (1;0,0,0) \\
Y_1 = C_0 \setminus ( L_0 \cup \{ P \} ) & [C_0] - m - 1 & (1,0,m) & (1,1,a+2) & (1;0,0,0) \\
Y_2 = L_0 \setminus C_0 & \LL + 1 - m & (0,0,m) & (1,a,a+2) & (1;0,0,0) \\
Y_3 = C_0 \cap L_0 & m & (1,0,m) & (1,a,a+2) & (1;0,0,0)
\end{array}
$$
Hence the contribution of $E_0 \setminus \{ P \}$ to the motivic zeta function $\Zmot_{,0}(D_1,D_2; s)$ is
\begin{equation}\label{eq:ZmotE0}
	\begin{aligned}
	\LL^{-3} & \frac{(\LL-1) \LL^{-(m s+a+2)}}{1-\LL^{-(m s+a+2)}} \Bigg(   \LL^2 - [C_0] + m  + ( [C_0] - m - 1 ) \, \frac{(\LL-1) \LL^{-( s+1)}}{1-\LL^{-( s+1)}}\\
	& + ( \LL + 1 - m ) \, \frac{(\LL-1) \LL^{-a}}{1-\LL^{-a}}  + m \, \frac{(\LL-1) \LL^{-( s+1)}}{1-\LL^{-( s+1)}}
	\cdot \frac{(\LL-1) \LL^{-a}}{1-\LL^{-a}}\Bigg).
	\end{aligned}
\end{equation}

The divisor $E_1$ is isomorphic to the weighted projective plane $\PP^2_w$ and it does contain
singular points of $Y$. Let us denote $C_1 = \widehat{D}_1 \cap E_1$ and by $L_x$, $L_y$, $L_z$ the
three coordinate axes of $E_1 = \PP^2_w$. Consider the following stratification of $E_1 = Y_{4} \sqcup
\dots \sqcup Y_{14}$, where $Y_4 = E_1 \setminus ( C_1 \cup L_x \cup L_y \cup L_z )$ and
$$
\begin{array}{l}
Y_5 = C_1 \setminus ( L_x \cup L_y \cup L_z ), \\
Y_6 = L_x \setminus ( C_1 \cup L_y \cup L_z ), \\
Y_7 = L_y \setminus ( C_1 \cup L_x \cup L_z ), \\
Y_8 = L_z \setminus ( C_1 \cup L_x \cup L_y ),
\end{array} \qquad
\begin{array}{l}
\ Y_9 = C_1 \cap L_x, \\
Y_{10} = C_1 \cap L_y, \\
Y_{11} = C_1 \cap L_z,
\end{array} \qquad
\begin{array}{l}
Y_{12} = L_x \cap L_y, \\
Y_{13} = L_x \cap L_z, \\
Y_{14} = L_y \cap L_z.
\end{array}
$$
In the table below the calculation for the strata $Y_7$, $Y_{10}$, $Y_{14}$ (resp.~$Y_{12}$)
was performed using the 1st chart (resp.~3rd chart) of $\pi_2$ where $x=0$ (resp.~$z=0$) is
the equation of the exceptional divisor. For the rest of strata the 2nd chart was used,
see~\eqref{eq:2nd-chart}. In order to simplify the table, denote $m_1 = \frac{pq}{k_1k_2}(m+k)$
and recall that $\nu_1 = \frac{kp}{k_1k_2} + \frac{kq}{k_1k_2} + \frac{pq}{k_1k_2}(a+2)$.
$$
\begin{array}{c|c|c|c|c}
\text{Stratum} & \text{Class} & \Nbf_{\ell} & \nubf_{\ell} & G_{\ell} \\
\hline
Y_4 & \begin{array}{c} \LL^2 - 2 \LL - [C_1] \\ + k_1 + k_2 + 2\\[0.1em] \end{array}
& (0,m_1,0) & (1,\nu_1,1) & (1;0,0,0) \\
Y_5 & [C_1] - k_1 - k_2 - 1 & (1,m_1,0) & (1,\nu_1,1) & (1; 0, 0, 0) \\
Y_6 & \LL - k_1 - 1 & (0,m_1,0) & (1,\nu_1,1) & (\frac{q}{k_2}; \frac{kp}{k_1k_2}, -1, 0) \\
Y_7 & \LL - k_2 - 1 & (m_1,0,0) & (\nu_1,1,1) & (\frac{p}{k_1}; -1, \frac{kq}{k_1k_2}, 0) \\
Y_8 & \LL - 2 & (0,m_1,m) & (1,\nu_1,a+2) & (\frac{k}{k_1k_2}; 0, -1, \frac{pq}{k_1k_2}) \\
\hline
Y_9 & k_1 & (1,m_1,0) & (1,\nu_1,1) & (\frac{q}{k_2}; \frac{kp}{k_1k_2}, -1, 0) \\
Y_{10} & k_2 & (m_1,1,0) & (\nu_1,1,1) & (\frac{p}{k_1}; -1, \frac{kq}{k_1k_2}, 0) \\
Y_{11} & 1 & (1,m_1,m) & (1,\nu_1,a+2) & (\frac{k}{k_1k_2}; 0, -1, \frac{pq}{k_1k_2}) \\
Y_{12} & 1 & (0,0,m_1) & (1,1,\nu_1) & (\frac{pq}{k_1k_2}; \frac{kp}{k_1k_2}, \frac{kq}{k_1k_2}, -1) \\
Y_{13} & 1 & (0,m_1,m) & (1,\nu_1,a+2) & (\frac{kq}{k_1k_2}; \frac{kp}{k_1k_2}, -1, \frac{pq}{k_1k_2}) \\
Y_{14} & 1 & (m_1,0,m) & (\nu_1,1,a+2) & (\frac{kp}{k_1k_2}; -1, \frac{kq}{k_1k_2}, \frac{pq}{k_1k_2}) \\
\end{array}
$$
The contribution of $E_1$ to the motivic zeta function $\Zmot_{,0}(D_1,D_2; s)$ can be written as
\begin{equation}\label{eq:ZmotE1}
\LL^{-3} \frac{(\LL-1) \LL^{-(m_1 s+\nu_1)}}{1-\LL^{-(m_1 s+\nu_1)}}\cdot Z(s),
\end{equation}
where $Z(s)$ is given by
\begin{equation*}
\begin{aligned}
	Z(s) = & \ \LL^2 - 2 \LL - [C_1] + k_1 + k_2 + 2 + ( \LL - k_1 - 1 ) S_{G_6} ( \Nbf_6, \nubf_6;s) \\[0.25cm]
	& + ( \LL - k_2 - 1 ) S_{G_7} ( \Nbf_7, \nubf_7;s) + S_{G_{12}} ( \Nbf_{12}, \nubf_{12};s) \\
	& + \Big( [C_1] - k_1 - k_2 - 1 + k_1 S_{G_9} ( \Nbf_9, \nubf_9;s) + k_2 S_{G_{10}} ( \Nbf_{10}, \nubf_{10};s) \Big)
	\, \frac{(\LL-1) \LL^{-( s+1)}}{1-\LL^{-( s+1)}} \\
	& + \Big( (\LL-2) S_{G_8}(\Nbf_8,\nubf_8;s) + S_{G_{13}}(\Nbf_{13},\nubf_{13};s) + S_{G_{14}}(\Nbf_{14},\nubf_{14};s) \Big)
	\, \frac{(\LL-1) \LL^{-(m s+a+2)}}{1-\LL^{-(m s+a+2)}} \\
	& + S_{G_{11}} ( \Nbf_{11}, \nubf_{11};s) \, \frac{(\LL-1) \LL^{-( s+1)}}{1-\LL^{-( s+1)}}
	\cdot \frac{(\LL-1) \LL^{-(m s+a+2)}}{1-\LL^{-(m s+a+2)}}.
	\end{aligned}
\end{equation*}

The sum of \eqref{eq:ZmotE0} and \eqref{eq:ZmotE1} gives rise to the wanted zeta function $\Zmot_{,0}(D_1,D_2; s)$.
To illustrate how the terms $S_{G}(\Nbf,\nubf;s)$ are calculated in practice, note that
if $\Nbf = (N_1,N_2,N_3)$, $\nubf = (\nu_1,\nu_2,\nu_3)$, and $G = (d;a,b,c)$, then
this term is simply
\begin{equation}\label{eqn:SG_dim3}
	S_{G}(\Nbf,\nubf;s) = \sum_{i=0}^{d-1} \LL^{\frac{1}{d} \left( (N_1 \overline{ia} + N_2 \overline{ib}
+ N_3 \overline{ic})  s + (\nu_1 \overline{ia} + \nu_2 \overline{ib} + \nu_3 \overline{ic})\right)},
\end{equation}
where $\overline{u}$ denotes the class of $u$ modulo $d$ satisfying $0 \leq \overline{u} \leq d-1$,
see also~\eqref{eq:sum-dim2}.

Finally we compute the topological zeta function associated with $\Zmot_{,0}(D_1,D_2; s)$ above.
Recall that when $\LL \to 1$, the sum $S_G(\Nbf,\nubf;s) \to |G|$. Then equation~\eqref{eq:ZmotE0}
becomes
$$
\left( ( 1 - \chi(C_0) + m ) + ( \chi(C_0) - m - 1 ) \, \frac{1}{s+1} + ( 2 - m ) \, \frac{1}{a}
+ m \, \frac{1}{s+1} \, \frac{1}{a} \right) \, \frac{1}{ms+a+2},
$$
where $\chi(C_0) = -m^2+3m+(p-1)(q-1)$. The curve $C_1 = \widehat{D}_1 \cap E_1$ is given by $z^k + x^q + y^p$
in~$\PP^2_w$. This is isomorphic to $z^{k_1k_2} + x^{k_2} + y^{k_1}$ in $\PP^2_{(k_1,k_2,1)}$,
see~\cite{Dolgachev82} and~\cite[Proposition~2.1]{AMO14}. According to~\cite[Theorem~5.7]{CMO14},
the Euler characteristic of this curve is $\chi(C_1) = k_1 + k_2 + 1 - k_1 k_2$.
Then equation~\eqref{eq:ZmotE1} gives the following term in the
topological zeta function
\begin{equation*}
\begin{aligned}
& \left[ k_1 k_2  - k_1 \frac{q}{k_2} - k_2 \frac{p}{k_1} + \frac{pq}{k_1k_2}
+ \left( - k_1 k_2 + k_1 \frac{q}{k_2} + k_2 \frac{p}{k_1} \right) \frac{1}{s+1} \right. \\
& \quad + \left. \left( - \frac{k}{k_1k_2} + \frac{kq}{k_1k_2} + \frac{kp}{k_1k_2} \right) \frac{1}{ms+a+2}
+ \frac{k}{k_1k_2} \, \frac{1}{s+1} \, \frac{1}{ms+a+2} \right] \frac{1}{m_1s+\nu_1}.
\end{aligned}
\end{equation*}
Note that for all $a \geq 1$, the candidate poles $-\frac{a+2}{m}$ and $-\frac{\nu_1}{m_1}$ induce eigenvalues
of the characteristic polynomial of the monodromy given in~\eqref{eq:monodromy}. In this case, $D_2 = (a-1)L$
always provides an \emph{allowed} differential $3$-form in the sense of~\cite{NemethiVeys12},
see also the recent work~\cite{AG19}.

% ====================================================================
% Nonabelian groups
% ====================================================================
\bigskip
\section{Quotient singularities with nonabelian groups}\label{sc:nonabelian_groups}

In this section we will illustrate the techniques developed in Section~\ref{sc:abelian_quotient_sing} by presenting an example
of a family of tetrahedral singularities and computing some invariants of motivic nature. We would like to emphasize that the calculation of the motivic invariants of this section, as well as those of Section~\ref{sc:Yomdin_Qres}, would have been infeasible by using classical resolution of singularities.

The underlying group of a tetrahedral singularity is a type (C) group in the classification
of~\cite[Chapter XII]{MBD:book}. We will use the following notation.
Let $d,q$ be two integers such that $0\leq q<d$ and $\gcd(d,q)=1$
and let us denote by $\xi$ a fixed primitive $d$th-root of unity.
Consider the quotient surface singularity $X=\CC^3/G_{d,q}$,
where
$$
G_{d,q}=\left\langle A,B\right\rangle\subset\GL_3(\CC)
$$
with
\[
	A=\begin{pmatrix}
		1 & &\\ & \xi &\\ & & \xi^q
	\end{pmatrix}
	\quad\text{and}\quad
	B=\begin{pmatrix}
		0 & 1 & 0\\ 0 & 0 & 1\\ 1 & 0 & 0
	\end{pmatrix}.
\]
Note that $A^d=B^3=\id$.

The variety $X$ is always $\QQ$-Gorenstein and log terminal as any quotient singularity under the action
of a linear group of finite order, see e.g.~\cite[Section~3.17]{Kollar13}. Using Reid-Tai's
criterion~\cite{Reid80} one checks that $X$ is never terminal and it is Gorenstein and canonical if and only if
$G_{d,q} \subset\text{SL}_3(\CC)$, that is, $q=d-1$. In the latter case they are called \emph{trihedral}
singularities and were extensively studied by Ito in order to construct a crepant resolution of $X$ and
prove the McKay correspondence~\cite{Ito94,Ito95_1,Ito95_2,Ito01}.

\subsection{A description of \texorpdfstring{$G_{d,q}$}{Gdq}}\mbox{}

Note that any element in $G$ is of the form $B^kM$, where $k=0,1,2$ and $M \in G$ is a diagonal matrix,
since the set of diagonal matrices is invariant under conjugation by $B$. Also, the diagonal matrices of
$G$ are obtained from $A$ by multiplying and permuting its diagonal elements. Thus, $G$ can be described
as 
\begin{equation}\label{eq:G-decomp}
G=\set{D_{i,j,k}}_{i,j,k=0}^{d-1}\sqcup\set{BD_{i,j,k}}_{i,j,k=0}^{d-1}\sqcup\set{B^2D_{i,j,k}}_{i,j,k=0}^{d-1},
\end{equation}
where $D_{i,j,k} = \diag(\xi^{jq+k}, \xi^{i+kq}, \xi^{iq+j}) = A^i(B^{-1}A^jB)(B A^k B^{-1})$. To compute the order
of $G$, one needs to study whether $D_{i,j,k} = \id$. This is equivalent to solving the following linear system of
congruences modulo $d$:
\begin{equation}\label{eq:sys-cong}
\left\{
\begin{aligned}
jq + k & \equiv 0, \\
i + kq & \equiv 0, \\
iq + j & \equiv 0.
\end{aligned}
\right.
\end{equation}
Since the determinant associated with the previous system is $q^3+1$, there are $d'=\gcd(d,q^3+1)$ triples $(i,j,k)$
with $D_{i,j,k} = \id$. In fact, the solutions of~\eqref{eq:sys-cong} can be described as $i \equiv q^2 \ell d/d' \!\mod d$,
$j \equiv \ell d/d'\!\mod d$, and $k \equiv - q \ell d/d' \!\mod d$, where $\ell = 0,1,\ldots,d'-1$. As a consequence $|G|=3d^3/d'$.

Let us now study when this group is small. Let $H_{d,q}$ be the normal subgroup of $G_{d,q}$ generated by
the elements having eigenvalues $\lambda=1$ with multiplicity $2$. Using the previous decomposition of $G$
given in \eqref{eq:G-decomp} and the description of the solutions of the system \eqref{eq:sys-cong} provided above,
one checks that $H_{d,q} = \{ \diag(\xi^{(q^3+1)i},\xi^{(q^3+1)j},\xi^{(q^3+1)k}) \}_{i,j,k=0}^{d'-1}$.
By construction, the quotient $G_{d,q}/H_{d,q}$ is small and it is in fact isomorphic to $G_{d',q}$ under
$A \mapsto \diag(1,\xi^{d/d'},\xi^{(d/d')q})$ and $B \mapsto B$. In particular, $G_{d,q}$ is small if and only if
$d \mid q^3+1$.

Hereafter we will assume the condition $d \mid q^3+1$ and denote the group simply by $G$ dropping the subindices.
In this case $|G| = 3d^2$ and we have the partition
\begin{equation}\label{eqn:struc_Gdq}
G=\set{D_{i,j}}_{i,j=0}^{d-1}\sqcup\set{BD_{i,j}}_{i,j=0}^{d-1}\sqcup\set{B^2D_{i,j}}_{i,j=0}^{d-1},
\end{equation}
where $D_{i,j}$ denotes $D_{i,j,0}$.

% % OLD
% Let us first study when this group is small. Let $H_{d,q}$ be the normal subgroup of $G_{d,q}$ generated by
% the elements having eigenvalues $\lambda=1$ with multiplicity $2$. One checks that $H_{d,q} =
% \{ \diag(\xi^{(q^3+1)i},\xi^{(q^3+1)j},\xi^{(q^3+1)k}) \}_{i,j,k=0}^{d'-1}$, where $d'=\gcd(d,q^3+1)$.
% By construction, the quotient $G_{d,q}/H_{d,q}$ is small and it is in fact isomorphic to $G_{d',q}$ under
% $A \mapsto \diag(1,\xi^{d/d'},\xi^{(d/d')q})$ and $B \mapsto B$. In particular, $G_{d,q}$ is small if and only if $d \mid q^3+1$.
% Hereafter we will assume this condition and denote the group simply by $G$ dropping the subindices.
%
% Note that any element in $G$ is of the form $B^kM$, where $k=0,1,2$ and $M \in G$ is a diagonal matrix,
% since the set of diagonal matrices is invariant under conjugation by $B$. Also, the diagonal matrices of
% $G$ are obtained from $A$ by multiplying and permuting its diagonal elements.
% Thus, $G$ can be described by the following partition:
% \begin{equation}\label{eqn:struc_Gdq}
% G=\set{D_{i,j}}_{i,j=0}^{d-1}\sqcup\set{BD_{i,j}}_{i,j=0}^{d-1}\sqcup\set{B^2D_{i,j}}_{i,j=0}^{d-1},
% \end{equation}
% where $D_{i,j}=\diag(\xi^{jq},\xi^{i},\xi^{iq+j}) = A^i(B^{-1}A^jB)$. As a consequence $|G|=3d^2$.
% % OLD

\subsection{The singular locus of the quotient space \texorpdfstring{$X$}{X}}\mbox{}

The arithmetical constants $\alpha=\gcd(d,q+1)$, $\beta=\gcd(d,q^2-q+1)$, and $\gamma = \frac{\alpha \beta}{d}$
will play an important role in the following. One can also express\footnote{This is a consequence of the fact
that $q^2-q+1 \equiv 3 \mod (q+1)$ and $q^3+1\equiv(q+1)^3\mod 3$.}
$\gamma$ as
\[
\gamma = \gcd\left(d, q+1, q^2-q+1, \frac{q^3+1}{d}\right)
=\gcd\left(d, q+1, 3, \frac{q^3+1}{d}\right)
= \gcd\left(3, d, \frac{q^3+1}{d}\right).
\]
In particular, $\gamma$ can be either $1$ or $3$.
To understand the singular locus of $X$, we study whether the isotropy group of any $P \in \CC^3$ is non-trivial.
The (projection of the) origin is singular because it has non-trivial stabilizer, namely the whole group~$G$.
Consider the lines in $\CC^3$ defined by $L:y=z=0$ and
$L_k: x-y = \xi^{k \frac{d}{\alpha}}x-z = 0$, for $k \geq 0$.
The matrix $D_{i,0}$ fixes the points of the form $(a,0,0)$ for $i=0,\ldots,d-1$.
Also, $BD_{0,j}$ fixes the points $(a,a,\xi^{k\frac{d}{\alpha}}a)$ for $j\equiv-k\frac{d}{\alpha} \mod d$.
Then the projection of $L$ and $L_k$ (again denoted by $L$ and $L_k$) are singular in $X$.
It is left to the reader to check that the lines $L_k$ and $L_{k'}$ are the same in $X$
if and only if $k \equiv k' \mod \gamma$.

It remains to show that there are no more singular points other than the ones contained in $L$ and $L_k$.
If $D_{i,j} P = P$, then $[P] \in X$ is equivalent to an element in $L$. If $B^2 D_{i,j} P = P$,
then $P = D_{i,j}^{-1} B P = B D_{k,l} P$ for some $k,l$ and hence only the case $B D_{i,j} P = P$ has to be studied.
Note that $BD_{i,j}$ has $\lambda = 1$ as an eigenvalue only if $i+j \equiv 0 \mod {d/\alpha}$ and in such a case
$\ker(BD_{i,j}-\id) = \langle ( \xi^{iq+i+j},\xi^{iq+j},1 ) \rangle$. The equation
$A^{i}( \xi^{iq+i+j},\xi^{iq+j},1 )^t = \xi^{iq+i+j} (1,1,\xi^{-(i+j)})^t$ together
with the fact that $i+j$ is a multiple of $d/\alpha$ implies that $[P] \in X$ has to be
in one of the $L_k$'s.

Summarizing, we obtain that $\Sing X$, pictured in Figure~\ref{fig:blowup}, is composed by lines passing through the origin,
which induce the following stratification given by the isotropy groups:
\[
\Sing X = \set{O} \sqcup \big( L \setminus \{ O \} \big)
\sqcup \bigsqcup_{k=0}^{\gamma-1} \big( L_k \setminus \{ O \} \big).
\]
If $\gamma=1$, then $L_0 = L_1 = L_2 = \{ x=y=z \}$.

\subsection{Blowing-up at the origin and the new singular locus}\label{sec:blowing-up-Gdq}\mbox{}

Let $D_1$ and $D_2$ be two $\QQ$-Cartier divisors in $X$, defined by $D_1 : (xyz)^{N}=0$ and $D_2:(xyz)^{\nu-1}=0$
for $N\geq0$ and $\nu>0$. One of the goals of this example is to compute $\Zmot_{,0}(D_1,D_2; s)$.
Since the origin has nonabelian isotropy group, our formula in Theorem~\ref{thm:intro_main2} can not be used directly.
We will show that the blow-up at the origin gives rise to an embedded $\QQ$-resolution of $D_1+D_2 \subset X$.
This idea is related to Batyrev's canonical abelianization~\cite{Batyrev99, Batyrev00}, see also \cite{Pouyanne92}.

Let $\pi_1: \widehat{\CC}^3 \to \CC^3$ be the blowing-up at the origin. The group $G$ acts naturally on
$\widehat{\CC}^3$ and $\pi_1$ induces a morphism on the quotients $\pi: Y = \widehat{\CC}^3/G \to X$
which is an isomorphism outside $E = \pi^{-1}(0) \simeq \PP^2/G$.
Note that although $G$ is small on $\CC^3$, this is no longer the case on $\widehat{\CC}^3$ if $\beta \neq 1$,
since $D_{i,j} = \xi^{kq\frac{d}{\beta}} \id$ for $i \equiv qj$, $j \equiv k \frac{d}{\beta} \mod d$, $k=0,\ldots,\beta-1$
and then it fixes all points of $E$.

Let us denote by $[(u:v:w)]$ the class of $(u:v:w)$ in  $\PP^2 /G$. Consider the points
$P = [(1:0:0)]$ and $R_k = [(1:\varepsilon^k:\varepsilon^{2k})]$, $k=0,1,2$,
where $\varepsilon$ is a primitive $3$th root of unity.
Note that $R_0$, $R_1$, $R_2$ are all the same if and only if $3 \mid d$, since in this case
$\varepsilon = \xi^{d/3}$ is a primitive $3$th root of unity and $A^{d/3}=\diag(1, \varepsilon, \varepsilon^2)$
makes $(1:\varepsilon^k:\varepsilon^{2k})$ belong to the same orbit\footnote{The assumption $d \mid q^3+1$ with
$3 \mid d$ implies $q\equiv 2 \! \mod 3$ and hence $q+1 \equiv q^2-q+1 \equiv 0 \! \mod 3$.}.
In such a case, i.e.,~$3 \mid d$, the points $P_n = [(1:1:\xi^{n\frac{q^2-q+1}{3}})]$ will be relevant.
Finally consider also $D = \{ uvw= 0 \} \subset E$, it corresponds to the intersection of the
strict transform of $D_1+D_2$ with $E$.

If a point in $Y$ is singular due to the diagonal matrices, then it has to be in $D$.
If a point is fixed by a matrix $B^2 D_{k,l}$, then it is also fixed  by some matrix of type $BD_{i,j}$.
For the latter we study the corresponding eigenvalues and eigenvectors.
To do so let us fix $\lambda_{i,j}\in\CC^*$ such that $\lambda_{i,j}^3 = \xi^{(i+j)(q+1)}$,
then the eigenvalues of the matrix $BD_{i,j}$ are $\{ \lambda_{i,j} \varepsilon^k \}_{k=0}^2$
and the eigenspaces are of the form
$\ker(BD_{i,j}-\lambda_{i,j}\varepsilon^k\id) = \langle ( 1, \lambda_{i,j}\varepsilon^k,
\lambda_{i,j}^{-1}\xi^{(i+j)q}\varepsilon^{2k}) \rangle$. We just need to study whether
$P_{n,k}=\left( 1: \lambda_{n}\varepsilon^k: \lambda_{n}^{-1}\xi^{nq}\varepsilon^{2k}\right)$
with $\lambda_{n}^3=\xi^{n(q+1)}$ are equal to one another in~$E$.

Two cases arise. If $3\nmid d$, then $\xi^n=\xi^{3m}$ for some $m\in\ZZ$, since $3$ is
invertible modulo~$d$. Thus, the set of points above is described by
$\left\{ \left( 1: \xi^{m(q+1)} \varepsilon^k: \xi^{m(2q-1)}\varepsilon^{2k} \right) \right\}_{m,k\in\ZZ}$.
Applying the action of $D_{i,j}$ with $i\equiv jq-m(q+1)\mod d$ and $j\equiv m\mod d/\beta$,
it turns out that $[P_{n,k}] = R_k$.
If $3 \mid d$, then one can choose $\lambda_{n} = \xi^{n\frac{q+1}{3}}$ and $\varepsilon = \xi^{d/3}$,
and therefore $[A^{-n\frac{q+1}{3}-k\frac{d}{3}} P_{n,k}] = P_{-n}$.
It is left to the reader to check that $[P_{n}]=[P_{n'}]\in E$ if and only if $n\equiv n'\mod 3$.

Summarizing, we have just shown that (we will see below that the equality holds)
\begin{enumerate}
\item if $3\nmid d$, then $E \cap \Sing Y \subseteq \{ R_0, R_1, R_2 \} \sqcup \big( D \setminus \{ P \} \big) \sqcup \{ P \}$;
\item if $3\mid d$, then $E \cap \Sing Y \subseteq \{ P_0 = R_0, P_1, P_2 \} \sqcup \big( D \setminus \{ P \} \big) \sqcup \{ P \}$.
\end{enumerate}
It is worth noticing that, when $\gamma=3$, the point $P_k$ is precisely
$[(1:1:\xi^{k\frac{d}{\alpha}})]$, which corresponds to the direction given by
the singular line $L_k$ in $X$. The previous description is pictured in Figure~\ref{fig:blowup}.

\begin{figure}[ht]
	\centering
	\scalebox{0.65}{
		\begin{tikzpicture}
		\begin{scope}[xshift=-7cm]
		\begin{scope}[cm={1,1,-1,0,(0,0)}, x=-0.6cm,y=-1cm,z=1cm]
		    % CC3/Gdq
		    \node at (0,4.5,5) {\Large $\displaystyle X=\CC^3/G$};
		
		    % Axes
		    %\draw [->] (0,0,0) -- (4,0,0) node [left] {$x$};
		    %\draw [->] (0,0,0) -- (0,4,0) node [right] {$y$};
		    \draw [->] (0,0,0) -- (0,0,4) node [left] {$z$};
		
		    % Lines
		    \draw [very thick] (0,0,-2) -- (0,0,4.5);
		    \def\k{0.1}
		    \draw [very thick] (-2,-2,-2-2*\k) -- (5,5,5+5*\k);
		    \def\c1{2}
		    \draw [dashed, very thick] (-2*\c1,-2*\c1,-3/2*\c1) -- (4*\c1,4*\c1,3*\c1);
                    \def\c2{2.3}
		    \draw [dashed, very thick] (-2*\c2,-2*\c2,-2.5/2*\c2) -- (4*\c2,4*\c2,2.5*\c2);
		
		    \node[right] at (0,0,4.5) {\large $L$};
		    \node[left] at (5,5.2,5.8) {\large $L_0$};
		    \node at (8,8,6.3) {\large $L_1$};
		    \node at (8,8.6,5.3) {\large $L_2$};
		
		    \node at (0,0,0) {\Large $\bullet$};
		    \node at (0,-0.3,0.3) {\large $O$};
		\end{scope}
		\end{scope}
		
		\begin{scope}[xshift=7cm, yshift=0cm]
		   \definecolor{bluegreen2}{RGB}{0, 85, 127}
		    \colorlet{colsing}{bluegreen2}
		    \colorlet{colsingtext}{colsing}
		
		    % line
		    \draw[<-, thick] (-8,0) -- (-6,0) node[midway,above] {\Large $\pi$};;
		
		       % CC3/Gdq
		      \node at (6.2,5) {\Large $\displaystyle Y=\widehat{\CC}^3/G$};
		
			% Rectangle
			\draw[ultra thick] (-6,-1.5) -- (-3,2) -- (5,2) -- (2.5,-1.5) -- cycle;
			
		% 	\draw[dashed] (-4,0) -- (4,0);
		% 	\draw[dashed] (1,2) -- (-1,-2);
			
			% Axis
			\draw[dashed, ultra thick, color=colsing] (1,2) -- (3.5,0);
			\node[right, xshift=5, text=colsingtext] at (3.5,0) {\Large $D: uvw=0$};
			\node[right, xshift=0, text=colsingtext] at (3.5,-0.9) {\normalsize $\displaystyle\Big(\frac{d}{\beta}\Big)$};
			
			% Intersec points axis
			\node[color=colsing] (P) at (1.5,1.55) {\Huge $\bullet$};
			\node[xshift=7, yshift=2, text=colsingtext] at (P.east) {\Large $P$};
			\node[xshift=-2, yshift=-15, text=colsingtext] at (P.south) {\normalsize $\displaystyle\Big(\frac{d^2}{\beta}\Big)$}; %\large $(pr)
			
			\draw[very thick] (P.center) -- (1.5,4.83);  \node at (1.7,5.4) {\Large $\widehat{L}$};
			\draw[very thick] (1.5,-1.5) -- (1.5,-2);
			
			\node[color=colsing] (P0) at (-3.5,0.5) {\Huge $\bullet$};
			\node[xshift=15, text=colsingtext] at (P0.north) {\Large $P_0$};
			\node[yshift=-5, text=colsingtext] at (P0.south) {\normalsize $(3)$};
			
			\draw[very thick] (P0.center) -- (-3.5,3.75); \node at (-3.6,4.3) {\Large $\widehat{L}_0$};
			\draw[very thick] (-3.5,-1.5) -- (-3.5,-3.2);
			
			\node[text=colsingtext] (P1) at (-2,0.8) {\Large $\circ$};
			\node[xshift=15, text=colsingtext] at (P1.north) {\Large $P_1$};
			\node[yshift=-5, text=colsingtext] at (P1.south) {\normalsize $(3)$};
			
			\draw[dashed, very thick] (P1.center) -- (-2,4.1);  \node at (-2.1,4.6) {\Large $\widehat{L}_1$};
			\draw[dashed, very thick] (-2,-1.5) -- (-2,-2.85);
			
			\node[text=colsingtext] (P2) at (-0.5,1.1) {\Large $\circ$};
			\node[xshift=15, text=colsingtext] at (P2.north) {\Large $P_2$};
			\node[yshift=-5, text=colsingtext] at (P2.south) {\normalsize $(3)$};
			
			\draw[dashed, very thick] (P2.center) -- (-0.5,4.4);   \node at (-0.6,4.9) {\Large $\widehat{L}_2$};
			\draw[dashed, very thick] (-0.5,-1.5) -- (-0.5,-2.5);
			
			\node[text=colsingtext] (R1) at (-1.5,-0.6) {\Large $\times$};
			\node[xshift=-7, text=colsingtext] at (R1.west) {\Large $R_1$};
			\node[xshift=0,yshift=-5, text=colsingtext] at (R1.south) {\normalsize $(3)$};
			
			\node[text=colsingtext] (R2) at (0.5,-0.2) {\Large $\times$};
			\node[xshift=-7, text=colsingtext] at (R2.west) {\Large $R_2$};
			\node[xshift=0,yshift=-5, text=colsingtext] at (R2.south) {\normalsize $(3)$};
			
			\node[right] at (2.5,2.75) {\Large $E \Big(\frac{3N}{\beta},\frac{3\nu}{\beta}\Big)\simeq \PP^2/G$};
		\end{scope}
		
		\end{tikzpicture}
	}
	\caption{The blowing-up $\pi: Y \to X$ and the singularities of $X$ and $Y$.
	} \label{fig:blowup}
\end{figure}
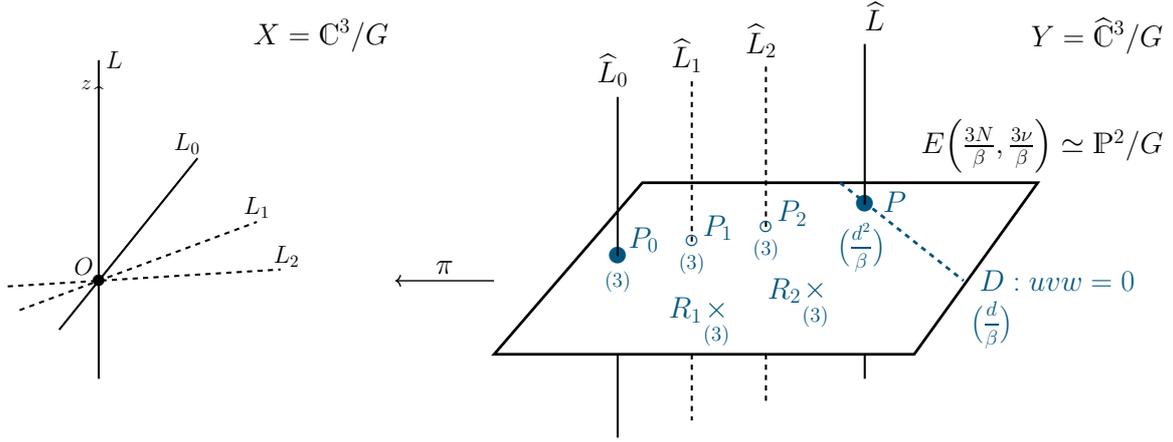

\subsection{Local study around the singular points and stratification}\mbox{}

It remains to determine the type of quotient singularity of each of the previous singular loci. The local study is performed
using the classical charts in $\widehat{\CC}^3$ and centering the origin at each of the points $(\CC^3/G_Q,0)\to (Y,Q)$.

In the case of $P$, this point is naturally the origin of one of the charts on the resolution space. Via the usual chart
$(x,y,z)\mapsto\left((x,xy,xz), (1:y:z)\right)$, one obtains that the action is linear and given by an abelian group of type
$$
\left(\begin{array}{c|ccc}
	d & 0 & 1 & q \\ d & q & 0 & q^2-q+1 \\
\end{array}\right).
$$
The underlying group is not small. In order to obtain a small group we use the isomorphism
$[(x, y, z)]\mapsto[(x^\beta, y, z)]$, giving
\begin{equation}\label{eqn:P_small}
	\left(\begin{array}{c|ccc}
		d & 0 & 1 & q\\ \frac{d}{\beta} & q & 0 & \frac{q^2-q+1}{\beta}\\
	\end{array}\right).
\end{equation}
This action is small and $(Y,P) \simeq \left(\CC^3/(C_{d}\times C_{d/\beta}), 0\right)$. In these new coordinates, the blowing-up is locally
described by $[(x^\beta, y, z)]\mapsto[(x,x y,x z)]$, with $K_{\pi}=\left(\frac{3}{\beta}-1\right) E$ giving
\[
	\pi^*D_1: x^{\frac{3N}{\beta}}y^Nz^N=0\quad\text{and}\quad \pi^*D_2+K_{\pi}: x^{\frac{3\nu}{\beta}-1}y^{\nu-1}z^{\nu-1}=0.
\]
Note that this is the only point in $Y$ where the action is not given by a cyclic group. For a point $Q\in D\setminus\set{P}$,
one sees that the action (\ref{eqn:P_small}) becomes of type $\left(\frac{d}{\beta}; q, 0, \frac{q^2-q+1}{\beta}\right)$.

Nevertheless, when we need to perform a translation to study other points, the action by the respective isotropy groups becomes non-linear. This issue is resolved by an adequate local analytic change of variables.

In the case of the points $R_k$, the induced isotropy group is $G_{R_k}=\{\xi^{i\frac{d}{\beta}}B^j\}_{i,j\in\ZZ}$.
In the new coordinates $(x,y,z)$ centered at $R_k$, the action becomes
\[
	(\xi^{i\frac{d}{\beta}}B^j)\cdot(x,y,z)=\left\lbrace\begin{array}{lcl}
		\displaystyle \left(\xi^{i\frac{d}{\beta}}x,y,z\right), & & \text{if }j\equiv0\mod 3, \\[1em]
		\displaystyle \left(\xi^{i\frac{d}{\beta}}x(y+\varepsilon^{k}),\frac{z-\varepsilon^{k}y}{y+\varepsilon^{k}},\frac{-\varepsilon^{2k}y}{y+\varepsilon^{k}}\right), & & \text{if }j\equiv1\mod 3, \\[1em]
		\displaystyle
		\left(\xi^{i\frac{d}{\beta}}x(z+\varepsilon^{2k}),\frac{-\varepsilon^{k}z}{z+\varepsilon^{2k}},\frac{y-\varepsilon^{2k}z}{z+\varepsilon^{2k}}\right), & & \text{if }j\equiv2\mod 3.
	\end{array}\right.
\]
One can check that the local analytic isomorphism $\varphi:(\CC^3,0)\to (\CC^3,0)$ given by $\varphi(x,y,z)=\left(xe^{y+z}, \varepsilon^{k}\big(e^{(\varepsilon-1)y+(\varepsilon^2-1)z}-1\big), \varepsilon^{2k}\big(e^{(\varepsilon^2-1)y+(\varepsilon-1)z}-1\big)\right)$
linearizes the action in $(\CC^3/G_{R_k},0)$, where the previous non-linear action of $\xi^{i\frac{d}{\beta}} B^j$
is transformed into multiplication by $\diag(\xi^{i\frac{d}{\beta}} \varepsilon^{kj},\varepsilon^j,\varepsilon^{2j})$.
Hence it gives a quotient singularity where the group is not small of type
$$
\left(\begin{array}{c|ccc}
d & \frac{d}{\beta} & 0 & 0\\ 3 & k & 1 & 2
\end{array}\right)
= \left(\begin{array}{c|ccc}
\beta & 1 & 0 & 0\\ 3 & k & 1 & 2
\end{array}\right)
\simeq \left(3 ;~k\beta,~1,~2\right),
$$
where the last morphism comes again form the isomorphism $[(x, y, z)]\mapsto[(x^\beta, y, z)]$. The strict transforms are expressed locally as
\[
	\pi^*D_1: x^{\frac{3N}{\beta}}u(y,z)=0\quad\text{and}\quad \pi^*D_2+K_{\pi}: x^{\frac{3\nu}{\beta}-1}v(y,z)=0,
\]
where $u(y,z)$ and $v(y,z)$ are units in the local ring. Similar arguments can be applied for the rest of the points.

Take the stratification $Y \cap E = \bigsqcup_{\ell \geq 0} Y_{\ell}$
of Theorem~\ref{thm:intro_main2}, where $Y_0 = E \setminus \left(  D\cup \bigcup_k Q_k \right)$ is the biggest stratum and
\[
	Q_i=\left\lbrace\begin{array}{lcl}
		R_k = \big(1:\varepsilon^k:\varepsilon^{2k}\big), & & \text{if } 3\nmid d, \\[0.5em] %
		P_k = \big(1:1:\xi^{k\frac{q^2-q+1}{3}}\big), & & \text{if } 3\mid d.
	\end{array}\right.
\]
We summarize the latter in the following table.
\[
	\begin{array}{l|c|c|c|c}
	\multicolumn{1}{c|}{\text{Stratum}} & \text{Class} & \Nbf_{\ell} & \nubf_{\ell} & G_{\ell} \\
	\hline &&&&\\[-0.35cm]
	Y_0 & [E] - [D] -3 & \left(\frac{3N}{\beta},0,0\right) & \left(\frac{3\nu}{\beta},1,1\right) & (1;0,0,0) \\[0.3cm]
	Y_1 = D \setminus P & [D] - 1 & \left(\frac{3N}{\beta},0,N\right) & \left(\frac{3\nu}{\beta}, 1, \nu\right) & \left(\frac{d}{\beta}; q, 0, \frac{q^2-q+1}{\beta}\right) \\[0.65em]
	Y_2 = P & 1 & \left(\frac{3N}{\beta},N,N\right) & \left(\frac{3\nu}{\beta}, \nu, \nu\right)  & \left(\begin{array}{c|ccc}
		d & 0 & 1 & q\\ \frac{d}{\beta} & q & 0 & \frac{q^2-q+1}{\beta}\\
	\end{array}\right) \\[1em]
	\hline
	\hspace*{-1ex} \begin{array}{l} Y_{3+k} = Q_k \\ k=0,1,2 \end{array}
	 & 1 & \left(\frac{3N}{\beta},0,0\right) & \left(\frac{3\nu}{\beta},1,1\right) & \begin{array}{ll} \\[-0.25cm]
			\text{if $3 \nmid d$:} &  (3; k\beta, 1, 2) \\[0.5em]
			\text{if $3 \mid d$:} &  \left(3; -k \frac{q+1}{\alpha} \gamma, 1, 2\right)
		\end{array}
	\end{array}
\]
In conclusion, $\pi:Y=\widehat{\CC}^3/G\to X$ is an embedded $\QQ$-resolution of singularities of the pair $(X,D_1+D_2)$, which is pictured in Figure~\ref{fig:blowup}.

% \begin{center}
% \begin{tikzcd}[column sep=.7em]
% 	 E \arrow[phantom,"\subset"]{r} \arrow[d] &  \widehat{\CC}^3 \arrow[rrrr, "\pi"] \arrow[d] & & & & \CC^3 \arrow[d] \\
% 	 E \arrow[phantom,"\subset"]{r} &  Y	\arrow[rrrr, "\pi"'] & & & & X
% \end{tikzcd}
% \end{center}

\subsection{Computation of the zeta functions and McKay correspondence}\mbox{}

Applying Theorem~\ref{thm:intro_main2} to compute $\Zmot_{,0}(D_1,D_2; s)$ using the previous stratification,
one obtains
\begin{equation}\label{eq:ZmotGdq}
\Zmot_{,0}(D_1,D_2; s)=\LL^{-3} \frac{(\LL-1) \LL^{-\frac{3}{\beta}(N s+\nu)}}{1-\LL^{-\frac{3}{\beta}(N s+\nu)}}
\cdot Z(s),
\end{equation}
where
\begin{equation*}\label{eq:ZmotGdq_exp}
\begin{aligned}
Z(s) & =  [E] - [D] -3 + S_{G_{3}}(\Nbf_3,\nubf_3;s)
+ S_{G_{4}}(\Nbf_4,\nubf_4;s) + S_{G_{5}}(\Nbf_5,\nubf_5;s) \\
& \quad + \frac{(\LL-1) \LL^{-(N s+\nu)}}{1-\LL^{-(N s+\nu)}}
\left( \left([D]-1\right)S_{G_1}(\Nbf_1,\nubf_1;s)
+ S_{G_2}(\Nbf_2,\nubf_2;s) \frac{(\LL-1) \LL^{-(N s+\nu)}}{1-\LL^{-(N s+\nu)}} \right).
\end{aligned}
\end{equation*}
For $\ell \neq 2$ the sum $S_{G_{\ell}}(\Nbf_\ell,\nubf_\ell;s)$ is computed as in~\eqref{eqn:SG_dim3},
since $G_\ell$ is cyclic. If $d \neq \beta$, then the group $G_2$ is noncyclic. Nevertheless,
$G_{2} = \{ M_{i,j} \}_{i,j}$ where the matrices $M_{i,j}$ can be written in terms of $\xi$ as
$$
M_{i,j}=\diag\left(\xi^{j\beta q}, \xi^{i},\xi^{iq+j(q^2-q+1)}\right),
\quad \begin{array}{ll} i=0,\ldots, d-1, \\ j=0,\ldots, d/\beta-1.\end{array}
$$
Strictly speaking one has to choose a $d^2/\beta$th root of unity to calculate $S_{G_2}(\Nbf_2,\nubf_2;s)$
in Theorem~\ref{thm:intro_main1}. Since all the terms appearing in $M_{i,j}$ above are $d$th roots of
unity, one gets
\[
S_{G_2}(\Nbf_2,\nubf_2;s) = \sum_{i=0}^{d-1}\sum_{j=0}^{d/\beta-1} \LL^{\frac{Ns+\nu}{d}
\left( \frac{3}{\beta} \overline{j\beta q} + \overline{\,i\,}  + \overline{(iq+j(q^2-q+1))} \right)},
\]
where $\overline{a}$ stands for the class of $a$ modulo $d$ satisfying $0 \leq \overline{a} \leq d-1$.

Specializing by the Euler characteristic in (\ref{eq:ZmotGdq}), we obtain the local topological zeta function.
Since $E \simeq \PP^2/G$ and $D$ is the quotient under $G$ of the three coordinate axes of~$\PP^2$,
one can employ the formula $\chi(S/G)=(1/|G|)\sum_{g\in G} S^g$, where $S^g$ is the set of points in $S$ fixed
by $g$, to compute its Euler characteristics. Using the discussion of Section~\ref{sec:blowing-up-Gdq},
one checks that $\chi(E) = 3$ and $\chi(D)=1$. Therefore
\begin{align*}
\Ztop_{,0}(D_1,D_2; s) & = \frac{\beta/3}{Ns+\nu}\left(8 + \frac{d^2/\beta}{(Ns+\nu)^2}\right)
= \frac{d^2+8\beta(Ns+\nu)^2}{3(Ns+\nu)^3}.
\end{align*}

The topological zeta function codifies the information of the so-called (local)
stringy Euler number $e_{\text{st},0}(X)$, obtained by substituting $s=0$ (or $N=0$) and $\nu=1$ in $\Ztop_{,0}(D_1,D_2; s)$,
\[
e_{\text{st},0}(X)=\frac{d^2+8\beta}{3}.
\]
By \cite[Theorem~3.6]{DL02}, cf.~equation~\eqref{eq:orb-measure} in Remark~\ref{rk:comparison-DL},
this is precisely the number of conjugacy classes of $G$, since $e_{\text{st},0}(X) = \chi(\mu^{\Gor}(\L(X)_0)) \in \ZZ$.
This formula was pointed out by Ito~\cite[Section~5]{Ito94} for the Gorenstein case, i.e.~$q=d-1$.
Moreover, she found a crepant resolution $h:\widetilde{X} \to X$ (i.e., with $K_h=0$) and proved the McKay correspondence,
that is, $\chi(\widetilde{X})$ is the number of conjugacy classes of $G$.
Note that our approach also implies the correspondence, since the group is small and from the change of variables formula as well as the fact that
$\chi(X\setminus O)=0$, it follows that $\chi(\widetilde{X}) =
\chi(\mu^{\Gor}(\L(X)^{\reg})) =
 \chi(\mu^{\Gor}(\L(X)_0^{\reg}))=e_{\text{st},0}(X)$.

% ====================================================================
% References
% ====================================================================
\providecommand{\bysame}{\leavevmode\hbox to3em{\hrulefill}\thinspace}


\begin{thebibliography}{10}
\bibitem{ACampo75}
N.~A'Campo, \emph{La fonction z\^eta d'une monodromie}, Comment. Math. Helv.
  \textbf{50} (1975), 233--248.

%\bibitem{ACLM05}
%E.~Artal~Bartolo, P.~Cassou-Nogu\`es, I.~Luengo, and A.~Melle~Hern\'andez,
%  \emph{Quasi-ordinary power series and their zeta functions}, Mem. Amer. Math.
%  Soc. \textbf{178} (2005), no.~841, vi+85.

\bibitem{AG19}
E.~{Artal Bartolo} and M.~{Gonz{\'a}lez-Villa}, \emph{{On Maximal order poles
  of generalized topological zeta functions}}, arXiv e-prints (2019),
  arXiv:1902.09815.

\bibitem{AMO14}
E.~Artal~Bartolo, J.~Mart\'{\i}n-Morales, and J.~Ortigas-Galindo, \emph{Cartier
  and {W}eil divisors on varieties with quotient singularities}, Internat. J.
  Math. \textbf{25} (2014), no.~11, 1450100, 20.

\bibitem{AMO14b}
E.~Artal~Bartolo, J.~Mart\'{\i}n-Morales, and J.~Ortigas-Galindo,
\emph{Intersection theory on abelian-quotient {$V$}-surfaces and {$\bf
Q$}-resolutions}, J. Singul. \textbf{8} (2014), 11--30.

\bibitem{Batyrev99}
V.~V.~Batyrev, \emph{Non-{A}rchimedean integrals and stringy {E}uler numbers of
  log-terminal pairs}, J. Eur. Math. Soc. (JEMS) \textbf{1} (1999), no.~1,
  5--33.

\bibitem{Batyrev00}
V.~V.~{Batyrev}, \emph{{Canonical abelianization of finite group actions}},
  arXiv e-prints (2000), arXiv:math/0009043.

\bibitem{BoriesVeys16}
B.~Bories and W.~Veys, \emph{Igusa's {$p$}-adic local zeta function and the
  monodromy conjecture for non-degenerate surface singularities}, Mem. Amer.
  Math. Soc. \textbf{242} (2016), no.~1145, vii+131.

\bibitem{BN16}
E.~{Bultot} and J.~{Nicaise}, \emph{{Computing motivic zeta functions on log
  smooth models}}, Math. Z. (2019). \texttt{https://doi.org/10.1007/s00209-019-02342-5}.

\bibitem{CauwbergsVeys17}
T.~Cauwbergs and W.~Veys, \emph{Monodromy eigenvalues and poles of zeta
  functions}, Bull. Lond. Math. Soc. \textbf{49} (2017), no.~2, 342--350.

\bibitem{ChNS18}
A.~Chambert-Loir, J.~Nicaise, and J.~Sebag, \emph{Motivic integration},
  Progress in Mathematics, vol. 325, Birkh\"auser, Basel, Switzerland., 2018.

\bibitem{CMO14}
J.I. Cogolludo-Agust\'{\i}n and J.~Mart\'{\i}n-Morales, J.and Ortigas-Galindo,
  \emph{Local invariants on quotient singularities and a genus formula for
  weighted plane curves}, Int. Math. Res. Not. IMRN (2014), no.~13, 3559--3581.

\bibitem{Craw04}
A.~Craw, \emph{An introduction to motivic integration}, Strings and geometry,
  Clay Math. Proc., vol.~3, Amer. Math. Soc., Providence, RI, 2004,
  pp.~203--225.

\bibitem{DL92}
J.~Denef and F.~Loeser, \emph{Caract\'eristiques d'{E}uler-{P}oincar\'e,
  fonctions z\^eta locales et modifications analytiques}, J. Amer. Math. Soc.
  \textbf{5} (1992), no.~4, 705--720.

\bibitem{DL98}
\bysame, \emph{Motivic {I}gusa zeta functions}, J. Algebraic Geom. \textbf{7}
  (1998), no.~3, 505--537.

\bibitem{DL99}
\bysame, \emph{Germs of arcs on singular algebraic varieties and motivic
  integration}, Invent. Math. \textbf{135} (1999), no.~1, 201--232.

\bibitem{DL01}
\bysame, \emph{Geometry on arc spaces of algebraic varieties}, European
  {C}ongress of {M}athematics, {V}ol. {I} ({B}arcelona, 2000), Progr. Math.,
  vol. 201, Birkh\"{a}user, Basel, 2001, pp.~327--348.

\bibitem{DL02}
\bysame, \emph{Motivic integration, quotient singularities and the {M}c{K}ay
  correspondence}, Compositio Math. \textbf{131} (2002), no.~3, 267--290.

\bibitem{Dolgachev82}
I.~Dolgachev, \emph{Weighted projective varieties}, Group actions and vector
  fields ({V}ancouver, {B}.{C}., 1981), Lecture Notes in Math., vol. 956,
  Springer, Berlin, 1982, pp.~34--71.

\bibitem{GLM97}
S.~M. Gusein{-}Zade, I.~Luengo, and A.~Melle{-}Hern{\'a}ndez, \emph{Partial
  resolutions and the zeta-function of a singularity}, Comment. Math. Helv.
  \textbf{72} (1997), no.~2, 244--256.

\bibitem{Igusa74}
J.~Igusa, \emph{Complex powers and asymptotic expansions. {I}. {F}unctions of
  certain types}, J. Reine Angew. Math. \textbf{268/269} (1974), 110--130,
  Collection of articles dedicated to Helmut Hasse on his seventy-fifth
  birthday, II.

\bibitem{Ito94}
Y.~Ito, \emph{Crepant resolution of trihedral singularities}, Proc. Japan
  Acad. Ser. A Math. Sci. \textbf{70} (1994), no.~5, 131--136.

\bibitem{Ito95_2}
\bysame, \emph{Crepant resolution of trihedral singularities and the orbifold
  {E}uler characteristic}, Internat. J. Math. \textbf{6} (1995), no.~1, 33--43.

\bibitem{Ito95_1}
\bysame, \emph{Gorenstein quotient singularities of monomial type in dimension
  three}, J. Math. Sci. Univ. Tokyo \textbf{2} (1995), no.~2, 419--440.

\bibitem{Ito01}
\bysame, \emph{Mc{K}ay correspondence and {$T$}-duality}, Proceedings of the
  {W}orkshop ``{A}lgebraic {G}eometry and {I}ntegrable {S}ystems related to
  {S}tring {T}heory'' ({K}yoto, 2000), no. 1232, 2001, pp.~88--100.

\bibitem{Kontsevich95}
M.~Kontsevich, \emph{{L}ecture at {O}rsay},  (December 7, 1995).

%\bibitem{Loeser88}
%F.~Loeser, \emph{Fonctions d'{I}gusa {$p$}-adiques et polyn\^omes de
%  {B}ernstein}, Amer. J. Math. \textbf{110} (1988), no.~1, 1--21.

%\bibitem{LoeserSebag}
%F.~Loeser and J.~Sebag, \emph{Motivic integration on smooth rigid varieties and
%  invariants of degenerations}, Duke Math. J. \textbf{119} (2003), no.~2,
%  315--344.

\bibitem{Kollar13}
J.~Koll\'{a}r, \emph{Singularities of the minimal model program},
  Cambridge Tracts in Mathematics, vol. 200, Cambridge University Press,
  Cambridge, 2013, With a collaboration of S\'{a}ndor Kov\'{a}cs.

\bibitem{Looijenga02}
E.~Looijenga, \emph{Motivic measures}, Ast\'erisque (2002), no.~276, 267--297,
  S\'eminaire Bourbaki, Vol. 1999/2000.

\bibitem{Martin11}
J.~Mart\'in-Morales, \emph{Monodromy zeta function formula for embedded
  {$\mathbb{Q}$}-resolutions}, Rev. Mat. Iberoam. \textbf{29} (2013), no.~3,
  939--967.

\bibitem{Martin12}
\bysame, \emph{Embedded {$\mathbb{Q}$}-resolutions for {Y}omdin-{L}\^e surface
  singularities}, Israel J. Math. \textbf{204} (2014), no.~1, 97--143.

\bibitem{Martin16}
\bysame, \emph{Semistable reduction of a normal crossing
  {$\mathbb{Q}$}-divisor}, Ann. Mat. Pura Appl. (4) \textbf{195} (2016), no.~5,
  1749--1769.

\bibitem{MBD:book}
G.~A. Miller, H.~F. Blichfeldt, and L.~E. Dickson, \emph{Theory and
  applications of finite groups}, Dover Publications, Inc., New York, 1961.

\bibitem{NemethiVeys12}
A.~N\'emethi and W.~Veys, \emph{Generalized monodromy conjecture in dimension
  two}, Geom. Topol. \textbf{16} (2012), no.~1, 155--217.

\bibitem{Nicaise10}
J.~Nicaise, \emph{An introduction to {$p$}-adic and motivic zeta functions and
  the monodromy conjecture}, Algebraic and analytic aspects of zeta functions
  and {$L$}-functions, MSJ Mem., vol.~21, Math. Soc. Japan, Tokyo, 2010,
  pp.~141--166.

\bibitem{Pouyanne92}
N.~Pouyanne, \emph{Une r\'{e}solution en singularit\'{e}s toriques
  simpliciales des singularit\'{e}s-quotient de dimension trois}, Ann. Fac.
  Sci. Toulouse Math. (6) \textbf{1} (1992), no.~3, 363--398.

\bibitem{Reid80}
M.~Reid, \emph{Canonical {$3$}-folds}, Journ\'{e}es de {G}\'{e}ometrie
  {A}lg\'{e}brique d'{A}ngers, {J}uillet 1979/{A}lgebraic {G}eometry, {A}ngers,
  1979, Sijthoff \& Noordhoff, Alphen aan den Rijn---Germantown, Md., 1980,
  pp.~273--310.

%\bibitem{RodriguesVeys01}
%B.~Rodrigues and W.~Veys, \emph{Holomorphy of {I}gusa's and topological zeta
%  functions for homogeneous polynomials}, Pacific J. Math. \textbf{201} (2001),
%  no.~2, 429--440.

%\bibitem{Sebag}
%J.~Sebag, \emph{Int\'{e}gration motivique sur les sch\'{e}mas formels}, Bull.
%  Soc. Math. France \textbf{132} (2004), no.~1, 1--54.

\bibitem{Steenbrink77}
J.~H.~M. Steenbrink, \emph{Mixed {H}odge structure on the vanishing
  cohomology}, Real and complex singularities ({P}roc. {N}inth {N}ordic
  {S}ummer {S}chool/{NAVF} {S}ympos. {M}ath., {O}slo, 1976), Sijthoff and
  Noordhoff, Alphen aan den Rijn, 1977, pp.~525--563.

\bibitem{Veys97}
W.~Veys, \emph{Zeta functions for curves and log canonical models}, Proc.
  London Math. Soc. (3) \textbf{74} (1997), no.~2, 360--378.

\bibitem{Veys99}
\bysame, \emph{The topological zeta function associated to a function on a
  normal surface germ}, Topology \textbf{38} (1999), no.~2, 439--456.

\bibitem{Veys01}
\bysame, \emph{Zeta functions and ``{K}ontsevich invariants'' on singular
  varieties}, Canad. J. Math. \textbf{53} (2001), no.~4, 834--865.

\bibitem{Veys06}
\bysame, \emph{Arc spaces, motivic integration and stringy invariants},
  Singularity theory and its applications, Adv. Stud. Pure Math., vol.~43,
  Math. Soc. Japan, Tokyo, 2006, pp.~529--572.

\bibitem{WY15}
M.~Wood and T.~Yasuda, \emph{Mass formulas for local {G}alois representations
  and quotient singularities. {I}: a comparison of counting functions}, Int.
  Math. Res. Not. IMRN (2015), no.~23, 12590--12619.

\bibitem{WY17}
\bysame, \emph{Mass formulas for local {G}alois representations and quotient
  singularities {II}: dualities and resolution of singularities}, Algebra
  Number Theory \textbf{11} (2017), no.~4, 817--840.

\bibitem{Yasuda04}
T.~Yasuda, \emph{Twisted jets, motivic measures and orbifold cohomology},
  Compos. Math. \textbf{140} (2004), no.~2, 396--422.

\bibitem{Yasuda06}
\bysame, \emph{Motivic integration over {D}eligne-{M}umford stacks}, Adv. Math.
  \textbf{207} (2006), no.~2, 707--761.

\bibitem{Yasuda14}
\bysame, \emph{The {$p$}-cyclic {M}c{K}ay correspondence via motivic
  integration}, Compos. Math. \textbf{150} (2014), no.~7, 1125--1168.

\bibitem{Yasuda16}
\bysame, \emph{Wilder {M}c{K}ay correspondences}, Nagoya Math. J. \textbf{221}
  (2016), no.~1, 111--164.

\bibitem{Yomdin74}
Y.~Yomdin, \emph{Complex surfaces with a one-dimensional set of singularities},
  Sibirsk. Mat. \v Z. \textbf{15} (1974), 1061--1082, 1181.
\end{thebibliography}
\end{document}